\documentclass[11pt,reqno]{amsart}

\usepackage{amsthm}
\usepackage{amssymb}
\usepackage{latexsym}
\usepackage{multicol}
\usepackage{verbatim,enumerate}
\usepackage{accents}

\usepackage[usenames]{color}
\usepackage[colorlinks=true, linkcolor=blue, citecolor=blue, urlcolor=blue]{hyperref}
\usepackage{hyperref}
\usepackage{amsmath, amscd}

\advance\textwidth by 1.3in \advance\oddsidemargin by -.6in \advance\evensidemargin by -.6in
\DeclareMathOperator{\Ann}{Ann}
\DeclareMathOperator{\rank}{rank}
\DeclareMathOperator{\orb}{orb}
\DeclareMathOperator{\stab}{stab}

\def\Gaff{\widehat{\mathfrak g}} 
\def\G{\mathfrak g} 
\def\H{\mathfrak h} 
\def\Haff{\widehat{\mathfrak h}} 
\def\Naff{\widehat{\mathfrak n}} 
\def\Baff{\widehat{\mathfrak b}} 
\def\CG{\mathfrak{Cg}} 
\def\CN{\mathfrak{Cn}} 
\def\CH{\mathfrak{Ch}} 

\def\a{\alpha}
\def\l{\lambda}
\def\d{\delta}
\def\L{\omega}

\newtheorem*{cor}{Corollary}
\newtheorem*{lem}{Lemma}
\newtheorem*{prop}{Proposition}

\newtheorem{thm}{Theorem}

\theoremstyle{definition}
\newtheorem*{defn}{Definition}
\newtheorem*{thm*}{Theorem}

\theoremstyle{remark}
\newtheorem*{rem}{Remark}

\newtheorem*{warning}{Warning}

\newenvironment{pf}{\proof}{\endproof}
\newcounter{cnt}
\newenvironment{enumerit}{\begin{list}{{\hfill\rm(\roman{cnt})\hfill}}{%
\settowidth{\labelwidth}{{\rm(iv)}}\leftmargin=\labelwidth%
\advance\leftmargin by \labelsep\rightmargin=0pt\usecounter{cnt}}}{\end{list}} \makeatletter
\def\mydggeometry{\makeatletter\dg@YGRID=1\dg@XGRID=20\unitlength=0.003pt\makeatother}
\makeatother \theoremstyle{remark}

\numberwithin{equation}{section}

\let\bwdg\bigwedge
\def\bigwedge{{\textstyle\bwdg}}

\newcommand{\id}{\operatorname{id}}

\newcommand{\wt}{\operatorname{wt}}

\newcommand{\nc}{\newcommand}
\newcommand{\rnc}{\renewcommand}

\nc{\cal}{\mathcal} \nc{\goth}{\mathfrak} \rnc{\bold}{\mathbf}

\nc\bomega{{\mbox{\boldmath $\omega$}}} \nc\bpsi{{\mbox{\boldmath $\Psi$}}}
 \nc\balpha{{\mbox{\boldmath $\alpha$}}}
 \nc\bpi{{\mbox{\boldmath $\pi$}}}
 \nc\bvpi{{\mbox{\boldmath $\varpi$}}}

  \nc\bxi{{\mbox{\boldmath $\xi$}}}
\nc\bmu{{\mbox{\boldmath $\mu$}}} \nc\bcN{{\mbox{\boldmath $\cal{N}$}}} \nc\bcm{{\mbox{\boldmath $\cal{M}$}}} \nc\blambda{{\mbox{\boldmath
$\lambda$}}}\nc\bnu{{\mbox{\boldmath $\nu$}}}

\newcommand{\lie}[1]{\mathfrak{#1}}

\makeatletter
\def\section{\def\@secnumfont{\mdseries}\@startsection{section}{1}%
  \z@{.7\linespacing\@plus\linespacing}{.5\linespacing}%
  {\normalfont\scshape\centering}}
\def\subsection{\def\@secnumfont{\bfseries}\@startsection{subsection}{2}%
  {\parindent}{.5\linespacing\@plus.7\linespacing}{-.5em}%
  {\normalfont\bfseries}}
\makeatother

 \nc{\Hom}{\operatorname{Hom}}
  \nc{\mode}{\operatorname{mod}}
\nc{\End}{\operatorname{End}} \nc{\wh}[1]{\widehat{#1}} \nc{\Ext}{\operatorname{Ext}} \nc{\ch}{\text{ch}} \nc{\ev}{\operatorname{ev}}
\nc{\Ob}{\operatorname{Ob}} \nc{\soc}{\operatorname{soc}} \nc{\rad}{\operatorname{rad}} \nc{\head}{\operatorname{head}}

\def\gr{\operatorname{gr}}

\def\ann{\operatorname{Ann}}

 \nc{\Cal}{\cal} \nc{\Xp}[1]{X^+(#1)} \nc{\Xm}[1]{X^-(#1)}
\nc{\on}{\operatorname} \nc{\Z}{{\bold Z}} \nc{\J}{{\cal J}} \nc{\C}{{\bold C}} \nc{\Q}{{\bold Q}}

\nc{\N}{{\bold N}}  \nc\boa{\bold a} \nc\bob{\bold b} \nc\boc{\bold c} \nc\bod{\bold d} \nc\boe{\bold e} \nc\bof{\bold f} \nc\bog{\bold g}
\nc\boh{\bold h} \nc\boi{\bold i} \nc\boj{\bold j} \nc\bok{\bold k} \nc\bol{\bold l} \nc\bom{\bold m} \nc\bon{\mathbb n} \nc\boo{\bold o}
\nc\bop{\bold p} \nc\boq{\bold q} \nc\bor{\bold r} \nc\bos{\bold s} \nc\boT{\bold t} \nc\boF{\bold F} \nc\bou{\bold u} \nc\bov{\bold v}
\nc\bow{\bold w} \nc\boz{\bold z} \nc\boy{\bold y} \nc\ba{\bold A} \nc\bb{\bold B} \nc\bc{\mathbb C} \nc\bd{\bold D} \nc\be{\bold E} \nc\bg{\bold
G} \nc\bh{\bold H} \nc\bi{\bold I} \nc\bj{\bold J} \nc\bk{\bold K} \nc\bl{\bold L} \nc\bm{\bold M} \nc\bn{\mathbb N} \nc\bo{\bold O} \nc\bp{\bold
P} \nc\bq{\bold Q} \nc\br{\bold R} \nc\bs{\bold S} \nc\bt{\bold T} \nc\bu{\bold U} \nc\bv{\bold V} \nc\bw{\bold W} \nc\bz{\mathbb Z} \nc\bx{\bold
x} \nc\KR{\bold{KR}} \nc\rk{\bold{rk}} \nc\het{\text{ht }}

\nc\toa{\tilde a} \nc\tob{\tilde b} \nc\toc{\tilde c} \nc\tod{\tilde d} \nc\toe{\tilde e} \nc\tof{\tilde f} \nc\tog{\tilde g} \nc\toh{\tilde h}
\nc\toi{\tilde i} \nc\toj{\tilde j} \nc\tok{\tilde k} \nc\tol{\tilde l} \nc\tom{\tilde m} \nc\ton{\tilde n} \nc\too{\tilde o} \nc\toq{\tilde q}
\nc\tor{\tilde r} \nc\tos{\tilde s} \nc\toT{\tilde t} \nc\tou{\tilde u} \nc\tov{\tilde v} \nc\tow{\tilde w} \nc\toz{\tilde z} \nc\woi{w_{\omega_i}}
\begin{document}
\nc\chara{\operatorname{Char}}


\title{Weyl modules for the hyperspecial current algebra}
\author{Vyjayanthi Chari}
\thanks{V.C. was partially supported by DMS-0901253 and DMS- 1303052}
\address{Department of Mathematics, University of California, Riverside, CA 92521}
\email{chari@math.ucr.edu}
\author{Bogdan Ion}
\thanks{B.I. was partially supported by  CNCS-UEFISCDI project PN II-ID-PCE-2011-3-0039.}
\address{Department of Mathematics, University of Pittsburgh, Pittsburgh, PA 15260}
\address{Algebra and Number Theory Research Center, Faculty of Mathematics and Computer Science, University of Bucharest,  14 Academiei St., Bucharest, Romania}
\email{bion@pitt.edu}
\author{Deniz Kus}
\address{Mathematisches Institut, Universit\" at zu K\" oln, Germany}
\email{dkus@math.uni-koeln.de}
\thanks{D.K. was partially supported by the “SFB/TR 12-Symmetries and
Universality in Mesoscopic Systems”.}
\date{\today}
\begin{abstract}
We develop the theory of global and local Weyl modules for the hyperspecial maximal parabolic subalgebra of type $A_{2n}^{(2)}$. We prove that the dimension of a local Weyl module depends only on its highest weight, thus establishing a freeness result for global Weyl modules. Furthermore, we show that the graded local Weyl modules are level one Demazure modules for the corresponding affine Lie algebra. In the last section we derive the same results for the special maximal parabolic subalgebras of the twisted affine Lie algebras not of  type $A_{2n}^{(2)}$. 
\end{abstract}
\maketitle
\section{Introduction}

The interest in the representation theory of current algebras naturally originated in the context of the representation theory of quantized enveloping algebras of affine Lie algebras. However, further investigations revealed that the category of graded representations of the current algebra with finite dimensional homogeneous components  has an appealing structure which we will describe in what follows. 

Let $\G$ be a simple complex Lie algebra. The current algebra $\G[t]$ associated to $\G$ is defined as the Lie algebra of polynomial maps $\bc\to \G$  and can be identified with the complex vector space $\G\otimes \bc[t]$ with Lie bracket the $\bc[t]$-bilinear extension of the Lie bracket on $\G$. This Lie algebra is naturally graded and the category of graded representations with finite dimensional homogeneous components is not semi-simple. In fact, it contains many well-known indecomposable reducible representations such as  the Kirillov-Reshetikhin modules and the $\G$-stable Demazure modules associated to positive level integrable representations of $\Gaff$, the untwisted affine Lie algebra corresponding to $\G$.  

In \cite{CP01}  the authors introduced and studied two important families of universal  highest weight indecomposable modules for the current algebra. Fix $\H$ a Cartan subalgebra of $\G$. The first family,  indexed by dominant integral weights $\lambda$ of $\G$,  are the  global Weyl modules $W(\lambda)$ which are cyclic on a highest weight vector $w_\lambda$.  These modules admit a compatible $ \lie g[t]$--grading and the homogenous components are finite--dimensional.  Moreover, 
in addition to being a left $\G[t]$-module, the global Weyl module $W(\lambda)$ admits a graded right module structure with respect to a  graded commutative algebra $\ba_\lambda$. This algebra is defined to be  the symmetric algebra of $\H[t]$ modulo the annihilator of the element $w_\lambda$ and it was shown that $\ba_\lambda$ is a polynomial algebra on finitely many generators depending on $\lambda$.

The second family of modules is  indexed by dominant integral weights $\lambda$ of $\G$ and maximal {\em{not necessarily graded ideals}}  $\bi$ in $\ba_\lambda$. These are called the  local Weyl modules and are denoted $W(\lambda,\bi)$ (see \cite{CFK10} for this formulation).  The algebra $\ba_\lambda$ has a unique graded maximal ideal and hence the  corresponding local Weyl module denoted $W(\lambda,\bi_{\lambda,0})$  is also graded.  The local Weyl modules are universal highest weight objects in the category of finite dimensional representations of $\G[t]$. The fundamental facts in this context -- proved in \cite{CP01} for $A_1$,   \cite{CL06} for $\G$ of type $A_n$, in \cite{FoL07} for $\G$ simply laced algebras of rank at least two, and in \cite{N12} for $\G$ not simply laced -- are the following: all the local Weyl modules with highest weight $\lambda$ have the same dimension. 
 For $\G$ simply-laced $W(\lambda,\bi_{\lambda,0})$ is  isomorphic to the $\G$--stable Demazure modules $D(-\lambda)$ associated to level one integrable representations of $\Gaff$ and extremal weight $-\lambda$. The situation for $\G$ non--simply laced case is more complicated, the graded local Weyl modules only admit a filtration by level one Demazure modules.
In any case, it  follows that  the global Weyl module $W(\l)$ is free as a $\ba_\lambda$-module with rank equal to the  dimension of the local Weyl modules $W(\lambda, \bi_{\lambda,0})$. 

More recently, efforts were made to develop a corresponding theory for twisted loop algebras and the so-called twisted current algebras $\G[t]^\sigma$.  We shall discuss this in some detail since it is relevant for the primary focus of this paper.
 The loop algebra $L(\lie g)$ is defined analogously to the current algebra by replacing $\bc[t]$ with $\bc[t,t^{-1}]$. The twisted loop algebra and the  twisted current algebra are   defined  to be  the fixed point subalgebras of $L(\lie g)$ and $\G[t]$ under an automorphism $\sigma$ induced from a non-trivial diagram automorphism of $\G$. The development of this theory for the twisted loop algebras is   is the following: the theory of local Weyl modules of such algebras was developed in \cite{CFS08}.   On the other hand, the theory of global Weyl modules  for twisted loop algebras is developed in \cite{FMS} {\em assuming that $\lie g$ is not of type $A_{2n}^{(2)}$}.  The theory of global Weyl modules for special kinds of  equivariant map algebras is developed in \cite{FMSa} but this work  is contingent on the assumption that a certain group action is free and therefore it cannot be used to study the twisted current algebras.  This work also  does not deal with the case of $A_{2n}^{(2)}$.
The final remark in this direction is that the isomorphism  between graded local Weyl modules for the twisted current algebras and level one Demazure modules for the associated twisted affine Lie algebra was  established in \cite{FK11} except  in the case of  $A_{2n}^{(2)}$, where only partial results are obtained.

There is a second point of view that leads to the construction of current algebras associated to affine Lie algebras $\Gaff$. From this point of view, the current algebra associated to $\Gaff$ is essentially defined as the special maximal parabolic subalgebra of $\Gaff$. Aside from $\Gaff$ of type $A_{2n}^{(2)}$, this point of view produces the usual current algebras and twisted current algebras. For $\Gaff$ of type $A_{2n}^{(2)}$ there are two conjugacy classes of special maximal parabolic subalgebras: one of them has distinguished properties and it is called hyperspecial. The algebra $\lie{sl}_{2n+1}[t]^\sigma$ studied in \cite{CFS08, FK11} corresponds to the special non-hyperspecial maximal parabolic subalgebra of type $A_{2n}^{(2)}$.

The goal of this paper is to develop the theory for  twisted current algebras  and it is now clear that the most  challenging  case is that of $A_{2n}^{(2)}$.   The main body of the paper is devoted to  the study of of global and local Weyl modules for its hyperspecial maximal parabolic subalgebra, which we also call current algebra to avoid deviating from the established terminology. In Section \ref{nine} we treat the other twisted current algebras corresponding to the maximal parabolic subalgebra of $\hat{\lie g}$; however, in the case of $\lie{sl}_{2n+1}[t]^\sigma$  the results of Section \ref{nine} are still incomplete since the dimension of the local graded Weyl modules is still not known in general. We remark here  that one of our motivations for studying the hyperspecial case is the relationship between the characters of global Weyl modules and Macdonald polynomials established in \cite{CI13}.

The study of the current algebras and twisted current algebras associated to affine Lie algebras  not of type $A_{2n}^{(2)}$ ultimately relies on the understanding of the representation theory of the current algebra $\lie{sl}_2[t]$ that is, up to isomorphism, the only rank one current subalgebras that can appear in such a situation. On the other hand, in studying the hyperspecial current algebra associated to the affine Lie algebra of type $A_{2n}^{(2)}$ one has to deal with a genuinely new phenomenon as one encounters and is bound to understand the rank one hyperspecial current algebra of type $A_2^{(2)}$  considered here for the first time. Furthermore, the hyperspecial current algebra is not realized in the same fashion as the other twisted current algebras which also contributes to the technical difficulties of this case.

Let us describe our results for the hyperspecial current algebra of type $A_{2n}^{(2)}$. We denote by $\G$ the homogeneous component of degree zero of the current algebra (a simple Lie algebra of type $A_1$ or $C_n$, $n\geq 2$, according to its rank), and by $\CG$ the current algebra itself. Our main results are the following; we refer to Section \ref{prelim} and Section \ref{reps} for the precise definition of the ingredients.

\begin{thm}\label{alambda}  
Let $\lambda$ be a dominant integral weight of $\G$. The algebra $\ba_\lambda$ is a graded polynomial algebra in  variables $T_{i,r}$ of grade $2r$, $1\le i\le n$, $1\le r\le\lambda(\alpha_i^\vee)$.
\end{thm}

\begin{thm}\label{locWeyl=Dem} Let $\lambda$ be a dominant integral weight of $\G$. The graded local Weyl module $W(\lambda,\bi_{\lambda,0})$ and the Demazure module $D(-\l)$   are isomorphic as graded $\CG$-modules.
\end{thm}

\begin{thm}\label{mainthm} 
Let $\lambda$ be a dominant integral  weight of $\G$. For any maximal ideal $\bi$ of $\ba_\lambda$, we have  $$\dim W(\lambda,\bi)=\prod_{i=1}^n\binom{2n+1}{i}^{\lambda(\alpha_i^\vee)}.$$
\end{thm}

From Theorem \ref{alambda} and Theorem \ref{mainthm} we obtain the following.
\begin{thm}\label{freeness}
Let $\lambda$ be a dominant integral weight of $\G$. The global Weyl module $W(\lambda)$ is a free $\ba_\lambda$-module of rank equal to $$\prod_{i=1}^n\binom{2n+1}{i}^{\lambda(\alpha_i^\vee)}.$$
\end{thm}
In Section \ref{nine} we obtain, except for $\lie{sl}_{2n+1}[t]^\sigma$, the corresponding results for twisted current algebras: the  global Weyl module $W(\lambda)$ is free as a $\ba_\lambda$-module with rank the common dimension of the local Weyl modules $W(\lambda, \bi)$.

The validity of the freeness property for the global Weyl modules associated to $\CG$ points perhaps to the existence of a more subtle canonical context for such questions. To further stress this point, let us mention that the understanding that the hyperspecial standard maximal parabolic subalgebra of $A_{2n}^{(2)}$ is better behaved than the special standard maximal parabolic originated in \cite{CI13} where the main results of this paper are used. As explained in \cite{CI13}, for this particular choice of maximal parabolic a rather subtle BBG-type reciprocity property holds for the category of graded representations with finite dimensional homogeneous components,  connecting the global Weyl modules, local Weyl modules, the simple objects in the category, and their projective covers, and the characters of the local and global Weyl modules  are specialized Macdonald polynomials. The BGG reciprocity and the connection with Macdonald theory hold also for the current algebras associated to all the other idecomposable affine Lie algebras.


\section{Preliminaries}\label{prelim}
\subsection{} We denote the set of complex numbers by $\bc$ and, respectively, the set of integers, non--negative integers, and positive integers  by $\bz$, $\bz_+$, and $\bn$. Unless otherwise stated, all the vector spaces considered in this paper are $\bc$-vector spaces and $\otimes$ stands for $\otimes_\bc$.


\subsection{}
For a Lie algebra $\lie a$, we let $\bu(\lie a)$ be the universal enveloping algebra of $\lie a$.  If, in addition, $\lie a$ is $\mathbb Z_+$-graded, then $\bu(\lie a )$ acquires the unique compatible $\mathbb Z_+$-graded algebra structure. We shall be interested in $\mathbb Z$-graded  representations $V=\oplus_{r\in\mathbb Z} V[r]$ of  $\mathbb Z_+$-graded Lie algebras $\lie a=\oplus_{r\in\bz_+} \lie a[r]$. Clearly, $\lie a[0]$ -- the homogeneous  component of $\lie a$ of grade zero -- is a Lie subalgebra  of $\lie a$ and if $V$ is a $\mathbb Z$-graded representation, then every homogeneous component $V[r]$ is a $\lie a[0]$--module. A morphism between graded $\lie a$-representations is a grade preserving map of $\lie a$-modules.


\subsection{}  We refer to \cite{K90} for the general theory of affine Lie algebras. Throughout, $\widehat{A}$ will denote the indecomposable affine Cartan matrix of type $A_{2n}^{(2)}$, $n\geq 1$,  and $\widehat S$  will denote the corresponding Dynkin diagram with the labeling of vertices as in Table Aff2 from \cite[pg.54-55]{K90}. 
Let $A$ and $S$ be the Cartan matrix and Dynkin diagram (of type $A_1$ if $n=1$ and of type $C_n$ if $n\geq 2$) obtained from $\widehat S$ by dropping the zero node.

Let $\widehat{\lie g}$ and $\lie g$  be the  affine Lie algebra and the finite--dimensional  algebra associated to $\widehat A$ and $A$, respectively. We can and shall realize $\lie g$ as a subalgebra of $\widehat{\lie g}$. We fix $\lie h\subset \widehat{\lie h}$ Cartan subalgebras of $\lie g $ and respectively $\widehat{\lie g}$. We denote by $\widehat{R}$ and, respectively,  $R$ the set of roots of $\widehat{\lie g}$ with respect to $\widehat{\lie h}$, and the set of roots of $\lie g $ with respect to ${\lie h}$. We fix $\{\alpha_0,\dots,\alpha_n\}$  a basis  for $\widehat R$ such that $\{\alpha_1,\dots,\alpha_n\}$ is a basis for $R$. The corresponding sets of positive and negative roots are denoted as usual by $\widehat R^\pm$ and respectively $R^\pm$.    We fix  $d\in \widehat{\lie h}$ such that $\alpha_0(d)=1$ and $\alpha_i(d)=0$ for $i\neq 0$; $d$ is called the scaling element and it is unique modulo the center of $\widehat{\lie g}$. For $0\le i\leq n$, define  $\omega_i\in\Haff^*$ by $\omega_i(\a_i^\vee)=\d_{i,j}$, for $0\leq j\leq n$, and $\omega_i(d)=0$, where $\d_{i,j}$  is Kronecker's delta symbol. The element $\omega_i$ is the fundamental weight of $\Gaff$ corresponding to $\a_i^\vee$.

Let $R_\ell$ and $R_s$ denote  respectively the subsets of $R$ consisting of the long and short roots and denote by $R_\ell^\pm, R_s^\pm$ the corresponding subsets of positive and negative roots. For $n=1$, by convention, $R_\ell=R$ and $R_s=\emptyset$.

If $\d$ denotes the unique non-divisible positive imaginary root in $\widehat{R}$ then $\widehat{R}=\widehat{R}^+\cup\widehat{R}^-$, where $\widehat{R}^-=-\widehat{R}^+$, $\widehat R^+ =\widehat R^+_{\rm {re}}\cup \widehat R^+_{\rm{im}}$, and
$$\widehat R^+_{\rm{im}} =\mathbb N\delta,\quad \widehat{R}^+_{\rm{re}}= R^+\cup(R_s+\mathbb N\delta)\cup(R_\ell+2\mathbb N\delta)\cup\frac12(R_\ell+(2\mathbb Z_++1)\delta).$$ For $n=1$, by convention, $R_s+\mathbb N\delta=\emptyset$.

We also need to consider the set $$\widehat R_{\rm re}(\pm)=R^\pm\cup(R_s^\pm+ \bn\delta)\cup(R_\ell^\pm+2\mathbb N\delta)\cup\frac12(R_\ell^\pm+(2\mathbb Z_++1)\delta).$$
Remark that $\widehat R_{\rm re}(+)\cup \widehat R_{\rm re}(-)= \widehat{R}^+_{\rm{re}}\cup R^-$.


\subsection{} 

Let  $Q=\oplus_{i=1}^n \bz \a_i$ be the root lattice of $R$ and let $\bar{Q}=\oplus_{i=1}^{n-1} \bz \a_i\oplus \bz \frac 12\a_n$. The respective $\bz_+$-cones are  $Q^+=\oplus_{i=1}^n \bz_+ \a_i$ and $\bar{Q}=\oplus_{i=1}^{n-1} \bz_+ \a_i\oplus \bz_+ \frac 12\a_n$. The weight lattice of $R$ is denoted by $P$ and the cone of dominant weights is denoted by $P^+$. Note that $P=\bar{Q}$ but $P^+\neq \bar{Q}^+$. Let $\widehat W$ and $W$ be the Weyl groups of  $\widehat{\lie g}$ and $\lie g$, respectively.


\subsection{}
Given $\alpha\in \widehat R^+$ let $\widehat{\lie g}_\alpha\subset\widehat{\lie g} $ be the corresponding root space; note that  $\widehat{\lie g}_\alpha \subset \lie g$ if $ \alpha\in R$.  We  define several subalgebras of $\widehat{\lie g}$ that will be needed in what follows. Let $\widehat{ \lie b}$ be the Borel subalgebra corresponding to $\widehat{R}^+$, and let $\widehat{\lie n}^+$ be its nilpotent radical,
 $$\widehat{\lie b}=\widehat{\lie h}\oplus \widehat{\lie n}^+,\ \  \ \ \widehat{\lie n}^\pm =\oplus_{\alpha\in\widehat R^+}\widehat{\lie g}_{\pm \alpha}.$$ The subalgebras $\lie b$ and $\lie n^\pm$ of $\lie g$ are analogously  defined. 

 The \emph{hyperspecial} standard maximal parabolic subalgebra of $\widehat{\lie g}$ is the standard maximal parabolic subalgebra corresponding to the Dynkin diagram of ${\lie g}$.
 The hyperspecial standard maximal parabolic subalgebra contains the center of $\widehat{\lie g}$ as a direct factor and therefore we can split off the center and consider instead
$$\lie k=(\lie h\oplus\mathbb Cd)\oplus\widehat{\lie n}^+\oplus\lie n^-.$$ 
 
 The current algebra $\CG$ is defined to be the ideal of $\lie k$ defined as
 $$
 \CG=\lie h\oplus\widehat{\lie n}^+\oplus\lie n^-.
 $$
 \begin{rem} Our definition of the current algebra of type $A_{2n}^{(2)}$ is different from the notion of twisted current algebra of type  $A_{2n}^{(2)}$ that exists in the literature (for example, as in \cite{FK11}).
\end{rem}
\begin{warning}
The current algebra $\CG$ depends on the affine Lie algebra $\Gaff$ and not only on $\G$ as the notation suggests. In fact, associated to  $\G$ of type $A_1$ there are two current algebras, corresponding to the affine Lie algebras of type $A_1^{(1)}$ and $A_2^{(2)}$, and associated to $\G$ of type $C_n$, $n\geq 2$, there are three current algebras, 
corresponding to the affine Lie algebras of type $C_n^{(1)}$,  $A_{2n-1}^{(2)}$, and $A_{2n}^{(2)}$, each exhibiting different behaviour in some respects.
\end{warning}
 
 The current algebra  has a triangular decomposition $$\CG=\lie C\lie n^+\oplus \CH \oplus\lie C\lie n^-,$$ where     $$\CH=\CH_+\oplus \H, \ \ \ \CH_+=\bigoplus_{k>0}\widehat{\lie g}_{k\delta},\ \ \ \lie C\lie n^\pm=\bigoplus_{\alpha\in \widehat{R}_{\rm re}(\pm)}\widehat{\lie g}_{\pm\alpha}.$$ Note that $\CH$ is an abelian Lie subalgebra.


\subsection{}\label{ev0} 
 
 The element $d$ defines a $\mathbb Z_+$-graded Lie algebra structure on $\CG$: for $\alpha\in \widehat R$ we say that $\lie g_\alpha$ has grade $k$ if $$[d,x_\alpha]=kx_\alpha$$ 
 or, equivalently, if $\alpha(d)=k$. Remark that since $\delta(d)=2$ the eigenvalues of $d$ are all integers and if $\lie g_\alpha\subset\CG$, then the eigenvalues are non--negative integers.  With respect to this grading, the zero homogeneous component of the current algebra is  $\CG[0]= \lie g$ and the subspace spanned by the positive homogeneous components is an ideal denoted by $  \CG_+$. We have a short exact sequence of Lie algebras,
$$0\to \CG_+\to\CG \stackrel{\ev_0}{\longrightarrow} \lie g\to 0.$$
Note that this exact sequence is right-split but not left-split as a sequence of Lie algebras but it is a split sequence as a sequence of $\lie g$-modules.


 
\section{Current algebra representations} \label{reps}
\subsection{}

We start by recalling some standard notation and results  on the representation theory of $\lie g$ and then introduce the representations of $\CG$ that we will be concerned with. 

A $\lie g$--module $V$ is said to be a weight module if it is $\lie h$--semisimple, $$V=\oplus_{\mu\in\lie h^*}V_\mu,\ \ \ V_\mu=\{v\in V: hv=\mu(h)v,\ \ h\in\lie h\}.$$ Set $\wt V=\{\mu\in\lie h^*:V_\mu\ne 0\}$.  
Given $\lambda\in P^+$, let $V(\lambda)$ be the irreducible finite--dimensional $\lie g$--module generated by an element $v_\lambda$ with defining relations:$$\lie n^+ v_\lambda=0,\ \ \ hv_\lambda=\lambda(h)v_\lambda,\ \ \ x_{-\alpha}^{\lambda(\alpha^\vee)+1}v_\lambda=0,$$
for all $h\in\lie h$ and $\alpha\in R^+$.  We have $\wt V(\lambda)\subset\lambda- Q^+$. Any irreducible finite--dimensional  $\lie g$--module is isomorphic to $V(\lambda)$ for some $\lambda\in P^+$. 
{{Moreover, any locally finite--dimensional $\lie g$--module $V$ is isomorphic to a direct sum of modules $V(\mu)$, $\mu\in P^+$, and in particular $\wt V$ is a $W$--invariant subset of $\lie h^*$.}}

Since we have an inclusion of $\lie g\hookrightarrow\CG$ we can and do  regard all $\CG$--modules as $\lie g$--modules by restriction. In particular a weight module for $\CG$ is  one which is a weight module for $\lie g$. {Notice that if we regard $\CG$   as a $\lie g$--module via  the adjoint action then $\CG$ is a weight module and $\wt\CG\subset \bar Q$. A similar statement is true for the adjoint action of $\lie g$ on $\bu(\CG)$.}

\subsection{}  For $\lambda\in P^+$, let $P(\lambda)$ be the $\CG$ module induced from $V(\lambda)$, $$P(\lambda)= \bu(\CG)\otimes_{\lie g}V (\lambda).$$ Setting $p_\lambda=1\otimes v_\lambda$,   it is easily seen that $P(\lambda)$ is generated as $\CG$-module by the element $p_\lambda$ with the same relations as $v_\lambda$.  Moreover, if we regard $\bu(\CG_+)$ as a $\lie g$--module via the adjoint action, we have an isomorphism of $\G$-modules,  $$P(\lambda)\cong\bu(\CG_+)\otimes V(\lambda).$$
Hence  $P(\lambda)$ is a weight module and  $\wt P(\lambda)\subset \lambda-\bar Q$.
In addition, $P(\lambda)$  is a projective module in the category of $\CG$--modules which are locally finite--dimensional $\lie g$--modules.   If we declare the grade of $p_\lambda$ to be $r\in\mathbb Z$, then  $P(\lambda)$ is also a $\mathbb Z$--graded $\CG$--module and we denote this module by $P(\lambda,r)$.    A non--zero  graded component of $P(\lambda,r)$ has grade $m\ge r$ and the $r^{\rm th}$ graded piece is isomorphic to $V(\lambda)$ and  all graded components are finite--dimensional $\lie g$--modules. In particular this proves that $P(\lambda,r)$ and hence also $P(\lambda)$ is a sum  of finite--dimensional $\lie g$--modules.
 Denoting by $\ev_0^* V(\lambda)$ the $\CG$--module obtained by pulling back $V(\lambda)$ through the homomorphism $\ev_0: \CG\to\lie g$, we see that $\ev_0^* V(\lambda)$ is an irreducible quotient of  $P(\lambda)$ and that $P(\lambda)$ is the projective cover of $\ev_0^* V(\lambda)$ in the category of $\CG$--modules which are locally finite-dimensional $\lie g$--modules. 

\subsection{} Let $W(\lambda)$ be the maximal quotient of $P(\lambda)$ such that $\wt W(\lambda)\subset \lambda- \bar Q^+$.  Then $W(\lambda)$ is locally finite--dimensional and   the arguments of \cite{CFK10} prove that if we denote by $w_\lambda$ the image of $p_\lambda$, then $W(\lambda)$ is generated by the element $w_\lambda$ with  relations:$$\lie C\lie n^+ w_\lambda=0,\ \ \  hw_\lambda=\lambda(h)w_\lambda,\ \ \ x_{-\alpha}^{\lambda(\alpha^\vee)+1}w_\lambda=0,$$for all $h\in\lie h$ and $\alpha\in R^+$.  Again, it is clear that if we declare the grade of $w_\lambda$ to be $r$ then $W(\lambda)$ is a $\mathbb Z$--graded module which we denote as $W(\lambda,r)$. It can also be defined as   the maximal graded quotient of $P(\lambda,r)$ such that $\wt W(\lambda,r)\subset \lambda- \bar Q^+$.
Finally, note that  $\ev_0^* V(\lambda)$ is a quotient of $W(\lambda)$ proving that $W(\lambda)_\lambda\ne 0$.  The module $W(\l)$ is called a global Weyl module.

\subsection{} Since $\CH$ is an abelian Lie algebra and $[\CH_+,\lie C\lie n^+]\subset\lie C\lie n^+$, following \cite{CFK10}, one sees that there is well--defined  (graded)  right action of $\bu(\CH_+)$ on $W(\lambda)$ by:$$(yw_\lambda)u=yuw_\lambda,\ \ u\in \bu(\CH_+),\ \ y\in\bu(\CG).$$ This action commutes with the (graded) left action of $\bu(\CG)$ and makes $W(\lambda)$ a (graded) $(\CG,\CH_+)$--bimodule.
Set $$\ann(w_\lambda)=\{u\in \bu(\CH_+): uw_\lambda=w_\lambda u=0\}, $$ and observe that this is a (graded)  ideal in $\bu(\CH_+)$. Let  $\ba_\lambda$ be the corresponding  (graded)  quotient. The following result  is clear from the discussion.
\begin{prop}\label{elementary}  Let $\lambda\in P^+$. Then,
\begin{enumerit} \item[(i)]   The global Weyl module $W(\lambda)$ (resp. $W(\lambda,r)$, $r\in\mathbb Z$) is a $(\CG,\ba_\lambda)$-bimodule (resp. graded $(\CG,\ba_\lambda)$-bimodule). 
\item[(ii)]  For any $\mu\in P$ the subspace  $W(\lambda)_\mu$ is a right $\ba_\lambda$--submodule of $W(\lambda)$, and hence we have a decomposition $$W(\lambda)=\bigoplus_{\mu\in P} W(\lambda)_\mu,$$ as $(\lie h,\ba_\lambda)$--bimodules.
\item[(iii)] As an $\ba_\lambda$--module the subspace $W(\lambda)_\lambda=w_\lambda \ba_\lambda$ is free of rank one.
\end{enumerit}
 \hfill\qedsymbol\end{prop}

\subsection{} Given a maximal ideal $\bi$ of $\ba_\lambda$, define the local Weyl module $W(\lambda, \bi)$ as $$W(\lambda, \bi)= W(\lambda)\otimes_{\ba_\lambda} \ba_\lambda/\bi.$$  Proposition \ref{elementary}(ii) shows that  $W(\lambda, \bi)$ is a weight module and we have $$W(\lambda, \bi)=\bigoplus_{\mu\in P} W(\lambda)_\mu\otimes_{\ba_\lambda} \ba_\lambda/\bi.$$  Proposition \ref{elementary}(iii) gives $$\dim W(\lambda, \bi)_\lambda =1,$$ and hence $W(\lambda,\bi)$ has a unique irreducible quotient $V(\lambda,\bi)$.   We denote the image of $w_\lambda$ in $W(\lambda, \bi)$ by $w_{\lambda,\bi}$.   Suppose that $\tilde{\bi}$ is the preimage of $\bi$ in $\bu(\CH_+)$. Then it is clear that $$W(\lambda,\bi)\cong W(\lambda)\otimes_{\bu(\CH_+)} \bu(\CH_+)/\tilde{\bi},$$ and hence $W(\lambda,\bi)$ is the quotient of $W(\lambda)$ obtained by imposing the additional relation,$$\tilde{\bi}w_{\lambda}=0.$$  The local Weyl modules are  graded if and only if the ideal $\tilde{\bi} $  or equivalently, $\bi$ is graded. We let $\tilde \bi_{\lambda,0}$ to be the augmentation  ideal in $\bu(\CH_+)$ and $\bi_{\lambda,0}$ to be its image in $\ba_\lambda$ in which case $W(\lambda,\bi_{\lambda,0})$ is graded and has  $\ev_0^* V(\lambda)$ as its graded quotient.
 Finally, we remark that 
 if $\lambda,\mu\in P^+$ and $\bi$ and $\bj$ are ideals in $\ba_\lambda$ and $\ba_\mu$ respectively, then $W(\lambda,\bi)\cong W(\mu,\bj)$ as $\CG$-modules if and only if  $\lambda=\mu$ and $\bi=\bj$. 

 
 \subsection{} \label{Demazuredefs}
 
Let $\widehat{V}(\omega_0)$ be  the unique irreducible  highest weight  $\Gaff$-module with highest weight $\omega_0$. For each $w\in \widehat{W}$ the weight space $\widehat{V}(\omega_0)_{w(\omega_0)}$ is one-dimensional. The $\Baff$-module generated by $\widehat{V}(\omega_0)_{w(\omega_0)}$  depends only on  $w(\omega_0)$ and is denoted by $D(w(\omega_0))$.  These modules  are  called level one Demazure modules and are finite dimensional $\Baff$-modules. 

The eigenspaces of the scaling element $d$ define a canonical $\bz$-grading  on the Demazure modules. Replacing $\omega_0$ with $\omega_0+s\frac{1}{2}\d$, $s\in\bz$,  produces modules  that differ only in the $d$ action: the grading of one is a shift  of the other. Since modulo $\delta$ we have $\widehat{W}(\omega_0)=P+\omega_0$, the level one Demazure modules are of the form $D(\lambda+s\frac{1}{2}\d+\omega_0)$ with $\lambda\in P$. To keep the notation as simple as possible, we will use $D(\lambda)$ to refer to $D(\lambda+\omega_0)$.

Although $D(\lambda)$ are by definition only $\Baff$-modules, for anti-dominant $\lambda$ they are admit a $\CG$-module structure. In this paper we will be concerned with the Demazure modules $D(-\lambda)$ for $\lambda\in P^+$.


\section{An explicit realization of the current algebra}\label{realization}

In this section we give an explicit construction of the algebra $\CG$ which will be used to prove our results. This can be found in \cite{Ca05} or \cite{K90}  but our exposition uses different notation and so we very briefly outline proofs for some of the statements. In the second half of this section we relate the local Weyl modules for $\CG$ to the local Weyl modules for the loop algebra of $\lie{sl}_{2n+1}$.

\subsection{}\label{kacrealize}  Let  $\lie{sl}_{2n+1}$ be the Lie algebra of trace zero $(2n+1)\times (2n+1)$ matrices with complex entries and  let $\lie h_{2n+1}$ be a fixed Cartan subalgebra.  We also fix  $\{\alpha(i): 1\le i\le 2n \}$ a basis for the root system of $\lie{sl}_{2n+1}$ with respect to $\lie h_{2n+1}$ and such that their labeling corresponds to the usual labeling for the Dynkin diagram of type $A_{2n}$. We denote by $\Phi^+=\{\alpha(i,j): 1\le i\le j\le 2n\}$ the corresponding set of positive roots where $\alpha(i,j)=\sum_{k=i}^j\alpha(k)$. We denote by $P_{2n+1}$ the weight lattice of $\lie{sl}_{2n+1}$ and by $P^+_{2n+1}$ the cone of dominant weights with respect to the fixed basis.

Let $\lie n^\pm_{2n+1}$ be the nilpotent subalgebra of $\lie{sl}_{2n+1}$ determined by $\pm\Phi^+$. 
We fix a Chevalley basis $\{X^\pm_{i,j}$, $H_i$ ~:~$1\le i\le j\le 2n\}$  for $\lie{sl}_{2n+1}$.
The assignment $$X^\pm_{i,i}\to X^\pm_{2n+1-i,2n+1-i}$$ extends to an algebra automorphism $\sigma$ of $\lie{sl}_{2n+1}$ and we set
\begin{align*}
x_0^\pm&=X^\mp_{1,n}-X^\mp_{n+1,2n},\\ 
x_i^\pm&= X^\pm_{i,i}+ X^\pm_{2n+1-i,2n+1-i},\ \ 1\le i\le n-1,\\ 
x_n^\pm&= X_{n,n+1}^\pm,\\
h_i&=H_i+H_{2n+1-i},\ \ 1\le i\le n.
\end{align*} 
It can be easily checked that the Lie algebra $\lie{sl}_{2n+1}$ is generated by $x_i^+$, $0\le i\le n$.

\subsection{} Let $\bc[t, t^{-1}]$ be the ring of Laurent polynomials in the  indeterminate $t$ and let $$L(\lie{sl}_{2n+1})=\lie{sl}_{2n+1}\otimes \bc[t, t^{-1}]$$ be the corresponding loop algebra with Lie bracket given by the $\bc[t, t^{-1}]$-bilinear extension of the Lie bracket of $\lie{sl}_{2n+1}$.

\begin{lem} The subalgebra of  $L(\lie{sl}_{2n+1})$ generated by the elements $x_i^\pm\otimes 1$, $1\le i\le n-1$ and $x_n^\pm\otimes t^{\pm 1}$ is isomorphic to $\lie g$.
\end{lem}
\begin{pf}  Recall that $\lie g$ is isomorphic to $\lie{sp}_{2n}$. The lemma is  now easily  verified by proving that the elements $x_i^\pm\otimes 1$, $x_n^\pm\otimes t^{\pm1 }$, $1\le i\le n-1$ and $h_j\otimes 1$, $1\le j\le n$ satisfy the same relations as the Chevalley generators for $\lie{sp}_{2n}$.
\end{pf}
Henceforth, we will identify $\G$ with the subalgebra of $L(\lie{sl}_{2n+1})$ described above. The following is proved in \cite[Chapter 8]{K90}. 
\begin{prop} The Lie subalgebra of $L(\lie{sl}_{2n+1})$ spanned by  the elements 
$$(X+(-1)^r\sigma(X))\otimes t^r,\ \  X\in\lie{sl}_{2n+1},\ \   r\in\mathbb{Z},$$
is isomorphic to $[\widehat{\lie g},\widehat{\lie g}]$ modulo its center.  Furthermore, $\CG$ is isomorphic to the subalgebra of $L(\lie {sl}_{2n+1})$ generated by 
$\G$
and $x_0^+$.
\end{prop}
In what follows we will identify $\CG$ with the subalgebra of $L(\lie {sl}_{2n+1})$ described above.


\subsection{} \label{handh}

We shall adopt the following conventions. We identify the subalgebra $\lie h$ of $\lie g$ with  the subspace of $\lie h_{2n+1}$ spanned by the elements $h_i$, $1\le i\le n$; the simple roots $\alpha_i$, $1\le i\le n$ with  the set of simple roots determined by the elements $x_i^+$, $1\le i\le n$ and the elements $\alpha_i^\vee$ with $h_i$, $1\le i\le n$.
We shall regard   $\lie h^*$ as a subspace of $\lie{h}_{2n+1}^*$ by extending $\mu\in\lie h^*$ as follows
$$\mu(H_i)=\mu(h_i),\ \ 1\le i\le n,\qquad  \mu(H_i)=0,\ \ i>n.$$  
Via this identification, the set of fundamental weights for $\lie{sl}_{2n+1}$ contains the fundamental weights for $\lie g$, allowing us to denote by $\omega_i$, $1\le i\le 2n$, the fundamental weights  for $\lie{sl}_{2n+1}$. We also have $P\subset P_{2n+1}$ and $P^+\subset P_{2n+1}^+$.


\subsection{}  

Let $\alpha\in R^+_s$ and $r\in \bz_+$;  $\alpha$ is necessarily of one of the two forms listed below for some $1\le i\le j< n$. We set
\begin{equation*}
\begin{aligned}
&x_{\pm\alpha+r\delta}= X^\pm_{i,j}\otimes t^r+(-1)^{i+j}X^\pm_{2n+1-j,2n+1-i}\otimes (-t)^r, &&\text{for } \alpha=\sum_{s=i}^j\alpha_s,\\
&x_{\pm\alpha+r\delta}=X^{\pm}_{i,2n-j}\otimes t^{r\pm 1}+(-1)^{i+j}X^{\pm}_{j+1,2n+1-i}\otimes (-t)^{r\pm 1}, &&\text{for }  \alpha=\sum_{s=i}^{j}\alpha_s + 2\sum_{s=j+1}^{n-1}\alpha_s+\alpha_n.
\end{aligned}
\end{equation*}
Let  $\alpha\in R^+_\ell$ and $r\in \bz_+$;  $\alpha$ is of the form $2(\alpha_{i}+\cdots +\alpha_n)-\alpha_n$ for some $1\le i\le n$. We set
\begin{equation*}
\begin{aligned}
& x_{\pm\alpha+2r\delta}= X_{i,2n+1-i}^\pm \otimes t^{2r\pm 1}, \\
&x_{\pm\frac12(\alpha+(2r+1)\delta)}= X^\pm_{i,n}\otimes t^{(2r+1\pm 1)/{2}}+(-1)^{i} X^\pm_{n+1, 2n+1-i}\otimes (-t)^{(2r+1\pm 1)/2}. 
\end{aligned}
\end{equation*}
Finally, for $1\le i\le n$,  we set $$h_{i,r\delta}= H_i\otimes t^r+H_{2n+1-i}\otimes (-t)^r.$$ We remark that $\alpha_i^\vee=h_{i,0}$ for $1\le i\le n$.
\begin{prop}\label{basis} The  set 
$$
\{x_{\alpha},~ h_{i,s\delta}: \alpha\in \widehat{R}_{\rm{re}}(\pm), 1\le i\le n, s\in\bz_+\}
$$ 
is a basis for $\CG$.
\end{prop}
\begin{pf} It is clear that the subspace spanned by the elements above is a subalgebra of $L(\lie{sl}_{2n+1})$ and since  $$x_i^\pm =x_{\pm\alpha_i}, \ \ 1\le i\le n-1,\ \ x_n^\pm\otimes t^{\pm 1}=x_{\pm\alpha_n},\ \ x_0^+=  x_{-\frac 12(2\sum_{i=1}^{n-1}\alpha_i+
\alpha_n)+\delta},$$  this subalgebra contains $\CG$. The claim then follows from the comparison of root spaces.
\end{pf}


\subsection{} The following is now clear.
\begin{prop}\label{sp2}
With the notation above, we have
\begin{enumerit} 
\item[(i)]
The set  $$\{x_{\alpha}: \alpha\in \widehat{R}_{\rm{re}}(\pm)\}$$ is a basis for $\lie C\lie n^\pm$ and the set $\{h_{i,r\delta}$~:~ $1\le i \le n$, $r\in\bn\}$ is a basis for $\CH_+$.
\item[(ii)] For $\alpha\in R^+$, the subalgebra of $\CG$ generated by the elements $$\{x_{\pm\alpha+r\delta}~:~r\in\mathbb{Z}_+, \a+r\delta\in \widehat{R} \}$$   is isomorphic to $\lie{sl}_2\otimes \bc[t]$.
\item[(iii)] For $\alpha\in R^+_\ell$,  the subalgebra generated by the elements $$\{x_{\frac 12(\pm\alpha+(2r+1)\delta)}, x_{\pm\alpha+2r\delta}~:~ r\in\mathbb{Z}_+\}$$ is isomorphic to the current algebra of type $A_2^{(2)}$.
\end{enumerit}
\end{prop}


\subsection{} \label{pullbacks} 

It is well--known (see \cite{CFK10} for instance) that any ideal of finite codimension in $L(\lie{sl}_{2n+1})$ is of the form $\lie{sl}_{2n+1}\otimes I$ where $I$ is an ideal in $\bc[t,t^{-1}]$ and 
$$
I\supseteq \left ((t-z_1)\cdots (t-z_k)\right)^N \bc[t,t^{-1}]
$$ 
for some non-zero complex numbers $z_s$, $1\le s\le k$ and $N\in\bn$.  

Given $\boz=(z_1,\dots,z_k)\in(\bc^\times)^k$ and $N\ge 1$, let $$\ev_{\boz, N}: L(\lie{sl}_{2n+1})\to\bigoplus_{s=1}^k\lie{sl}_{2n+1}^{z_s,N},\ \ \lie{sl}_{2n+1}^{z_s,N}=\lie{sl}_{2n+1}\otimes\frac{ \bc[t,t^{-1}]}{\left((t-z_s)^N\right)},$$  be the induced morphism of Lie algebras.   Let  
$$
\Psi_{\boz, N}:\CG\hookrightarrow L(\lie{sl}_{2n+1})\to \bigoplus_{s=1}^k\lie{sl}_{2n+1}^{z_s,N},
$$ 
be the restriction of  $\ev_{\boz, N}$ to $\CG$. 
\begin{lem}\label{generalev} 
If $z_p\ne z_s$ for $1\le s\neq p\le k$ then the morphism $\ev_{\boz, N}$ is surjective. If   $z_s\ne\pm z_p$ for $1\le s\ne p\le k$ then the morphism $\Psi_{\boz, N}$ is surjective.
\end{lem}

\begin{pf} 
The first statement follows from the fact that if   $z_s\ne z_p$ for $1\leq s\ne p\leq k$ then 
$$
\bc[t,t^{-1}]/(((t-z_1)\cdots(t-z_k))^N)\cong\bigoplus_{s=1}^k\bc[t,t^{-1}]/((t-z_s)^N).
$$ 
For the second claim, remark that, because of the above isomorphism, and the fact that 
$$
\bc[t,t^{-1}]/ (((t-z_1)\cdots (t-z_k))^N) \cong \bc[t]/ (((t-z_1)\cdots (t-z_k))^N),
$$
it is enough to show that 
$$
\CG\hookrightarrow L(\lie{sl}_{2n+1})\to \lie{sl}_{2n+1}\otimes \bc[t]/ (((t-z_1)\cdots (t-z_k))^N)
$$
is surjective. To avoid introducing more notation we will use $\Psi_{\boz, N}$ to denote the above map too. Let us denote $P(t)=((t-z_1)\cdots (t-z_k))^N$.  Because $P(t)$ and $P(-t)$ are, by hypothesis, relatively prime, there exist $A(t), B(t)\in \bc[t]$ such that
$$
A(t)P(t)+B(t)P(-t)=1.
$$
Fix $1\leq i< n$. For all $r\in \bz_+$, $x_{\pm\a_i+r\delta}=X^\pm_{i,i}\otimes t^r+X^\pm_{2n+1-i,2n+1-i}\otimes (-t)^r\in \CG$ so, for any $R(t)\in \bc[t]$, we have
$$
X^\pm_{i,i}\otimes R(t)+X^\pm_{2n+1-i,2n+1-i}\otimes R(-t)\in \CG.
$$
If $Q(t)\in \bc[t]$, then applying this for $R(t)=Q(t)B(t)P(-t)$ and respectively  for $R(t)=Q(-t)B(-t)P(t)$ we obtain that the class of $X^\pm_{i,i}\otimes Q(t)$ and $X^\pm_{2n+1-i,2n+1-i}\otimes Q(t)$ are the image of $\Psi_{\boz, N}$. We therefore obtain that the class of $X^\pm_{i,i}\otimes t^r$ is in the image of $\Psi_{\boz, N}$ for all $1\leq i\leq 2n, ~i\not \in \{n,n+1\}$, and all $r\in \bz_+$. Repeating this argument for $x_{\frac12(\pm\alpha_n+(2r+1)\delta)}\in \CG$, $r\in \bz_+$, we obtain that $X^\pm_{i,i}\otimes t^{(2r+1\pm 1)/2}$ is in the image of $\Psi_{\boz, N}$ for all $n\leq i\leq n+1$ and all $r\in \bz_+$. Since the degree $r$ component of $\lie{sl}_{2n+1}\otimes \bc[t]$ is generated as a $\G$-module by $X_{i,i}^\pm\otimes t^r$, $1\le i\le 2n$, we obtain that the map $\Psi_{\boz, N}$ is surjective. 
\end{pf}
\begin{rem}
Note that the proof of Lemma \ref{generalev} shows that the restriction of $\Psi_{\boz, N}$ to $\oplus_{s\geq r}\CG[s]$ remains surjective for any $r\in \bz_+$.
\end{rem}

\subsection{}\label{surj}
 For the rest of this section we shall relate the finite dimensional representation theory of $\CG$ and  $L(\lie{sl}_{2n+1})$. As a consequence of the results in Section \ref{pullbacks}, we see that  any finite--dimensional representation $V$ of $L(\lie{sl}_{2n+1})$  can be regarded as a module for $\oplus_{s=1}^k\lie{sl}_{2n+1}^{z_s,N}$ for some non--zero distinct complex numbers $z_1,\dots, z_s$ and some $N\in\bn$.  We denote by $\ev_{\boz, N}^*V$ and  $\Psi_{\boz, N} ^*V$  the representation of  $L(\lie{sl}_{2n+1})$,  and respectively of  $\CG$,  obtained by pulling back $V$ through the corresponding morphism. If in addition  $\ev_{\boz, N}^*V$ is a cyclic $L(\lie{sl}_{2n+1})$--module and $z_s\ne \pm z_p$ for $1\le s\ne p\le k$, then as a consequence of Lemma \ref{generalev} we obtain that $\Psi_{\boz, N} ^*V$ is also a cyclic $\CG$--module.  


\subsection{} For $\mu\in P^+$, recall that $V(\mu)$ denotes the irreducible highest weight representation of $\lie g$ with highest weight $\mu$. If $\lambda\in P_{2n+1}^+$, we denote by $V_{2n+1}(\lambda)$ the irreducible highest weight module for $\lie{sl}_{2n+1}$ with highest weight $\lambda$. 
If  $\lambda_s\in P^+_{2n+1}$ and $z_s\in\bc^\times$,  $1\le s\le k$, then the  canonical map $$\ev^*_{\boz,1}(V_{2n+1}(\lambda_1)\otimes\cdots\otimes V_{2n+1}(\lambda_k))\cong \ev_{z_1,1}^*V_{2n+1}(\lambda_1)\otimes \cdots \otimes \ev^*_{z_k, 1}V_{2n+1}(\lambda_k),$$ is an isomorphim of $L(\lie{sl}_{2n+1})$--modules.


  \subsection{} \label{untwistedweyl} We recall now the definition of the local Weyl modules for $L(\lie{sl}_{2n+1})$.
  For $1\le i\le 2n$, and $r\in\bz$, inductively define elements $\Lambda_{i,r}\in\bu(L(\lie h_{2n+1}))$  by 
  $$
  \Lambda_{i,r}=-\frac 1r\sum_{s=0}^{r-1}( H_i\otimes t^{s+1})\Lambda_{i,s},\ \  \Lambda_{i,0}=1.
  $$
  The elements $\Lambda_{i, r}$,  $H_i\otimes 1$, $1\le i\le 2n$, $r\in\mathbb{Z}$, are a set of polynomial generators for $\bu(L(\lie h_{2n+1}))$.

 Suppose that  $\bpi=(\pi_1(u),\dots,\pi_{2n}(u))$ is a collection of polynomials  in   an indeterminate   $u$ such that $\pi_i(u)=1+\sum_{r\ge 1}\pi_{i,r}u^r$ for some $\pi_{i,r}\in\bc$.  For  $1\le i\le 2n$ and $r\in\mathbb{Z}_+$ denote,
$$
 \sum_{r\ge 0}\pi_{i,-r}u^{r}=  \pi_{i,\deg \pi_i}^{-1}\sum_{r\ge 1}\pi_{i,r} u^{\deg\pi_i-r}.
 $$
 \begin{defn}
 Let $\bpi$ be as above. The local Weyl module $W_{2n+1}(\bpi)$ is the $L(\lie{sl}_{2n+1})$--module generated by an element $w_\bpi$ satisfying the relations: 
 \begin{equation*}
 \begin{aligned} 
 &L(\lie n^+_{2n+1})w_\bpi=0, && H_i w_\bpi=\deg\pi_i\cdot w_\bpi,\\
 & (X_i^-)^{\deg\pi_i+1}w_\bpi=0,&& \Lambda_{i, \pm r}w_\bpi=\pi_{i,\pm r} w_\bpi,
 \end{aligned}
 \end{equation*} 
 for all $1\le i\le 2n$ and $r\in\mathbb{Z}_+$.
\end{defn}
Parts (i) and (ii) of  the following Theorem are a combination of results proved in  \cite{CL06} and \cite{CP01}. Part (iii) was proved in \cite[Section 2.7.]{CFS08}.
\begin{thm}\label{charietal}
 Let $\bpi$ be as above. Then,
\begin{enumerit}
\item[(i)] We have, $$\dim W_{2n+1}(\bpi)=\prod_{i=1}^{2n}\binom{2n+1}{i}^{\deg\pi_i}.$$
\item[(ii)] There is a unique  irreducible quotient $V(\bpi)$ of $W(\bpi)$ and there is an isomorphism of $L(\lie{sl}_{2n+1})$--modules, $$V(\bpi)\cong \ev_{z_1,1}^*V_{2n+1}(\lambda_1)\otimes\cdots\otimes \ev_{z_k, 1}^* V_{2n+1}(\lambda_k),$$ where 
 $$\lambda_s=\sum_{i=1}^{2n}m_{i,s}\omega_i,\quad 1\leq s\leq k, \quad\text{and}\quad \pi_i=\prod_{s=1}^k(1-z_su)^{m_{i,s}}, \quad 1\leq i\leq 2n.$$
\item[(iii)] For $N\in\bn$ sufficiently large, we have  $$\left(\lie{sl}_{2n+1}\otimes\prod_{s=1}^k(t-z_s)^N \bc[t,t^{-1}]\right)W_{2n+1}(\bpi)=0.$$
\end{enumerit}
\end{thm}


\subsection{} We continue to use the notation used in Section \ref{untwistedweyl}. The following result will play an important role.
\begin{prop}\label{piw} 
Assume that $\bpi=(\pi_1,\dots, \pi_{2n})$  has the additional property that for any $z\in\bc$ at most one of $z$ or $-z$ is a root of $\pi_1\cdots\pi_{2n}$. Let 
$$\lambda=\sum_{i=1}^n(\deg \pi_i+\deg\pi_{2n+1-i})\omega_i\in P^+.$$
There exists a maximal ideal $\bi_\bpi$ of $\ba_\lambda$ such that  $\Psi_{\boz, N}^*W_{2n+1}(\bpi)$ is a quotient of $W(\lambda,\bi_\bpi)$. In consequence, $$\dim W(\lambda,\bi_\bpi)\ge\dim\Psi_{\boz, N}^*W_{2n+1}(\bpi)=\prod_{i=1}^{2n}\binom{2n+1}{i}^{\deg\pi_i}.$$ The unique irreducible quotient of $W(\lambda, \bi_\bpi)$ is 
isomorphic to 
$\Psi_{\boz, 1}^*(V(\lambda_1))\otimes\cdots\otimes \Psi_{\boz, 1}^*(V(\lambda_k))$.
\end{prop}
\begin{pf}
Form the definition of  $W_{2n+1}(\bpi)$ we see that $\bu(\CG)w_\bpi$ is a quotient of $W(\lambda)$. Moreover, $\dim \bu(\CH_+)w_\bpi=1,$ and hence it follows that there exists a maximal ideal $\bi_\bpi$ in $\ba_\lambda$ such that $$\bi_\bpi w_\bpi=0.$$ Therefore, there exists a surjective map of $\CG$--modules,  $$W(\lambda, \bi_\bpi)\to \bu(\CG)w_\bpi\to 0.$$
By Theorem \ref{charietal}(iii), there exist $N$ such that the action of $L(\lie{sl}_{2n+1})$ factors through $\bigoplus_{s=1}^k\lie{sl}_{2n+1}^{z_s,N}$. Our hypothesis allows us to use Lemma \ref{generalev} which, as explained in Section \ref{surj}, implies that  $\bu(\CG)w_\bpi=\Psi_{\boz, N}^*W_{2n+1}(\bpi)$. Our claim now immediately follows.
\end{pf}


\section{ The algebra  \texorpdfstring{$\ba_\lambda$}{A-lambda}} \label{five}

\subsection{}
The goal of this section is to give the proof of Theorem \ref{alambda}.

The proof proceeds as follows. We first define elements $P_{i,r}\in\bu(\CH_+)$ which are the analogs of the elements $\Lambda_{i,r}$ introduced in Section \ref{untwistedweyl} in the context of $L(\lie{sl}_{2n+1})$.  We then show that given $\lambda\in P^+$  all but finitely many of these elements are zero in $\ba_\lambda$. In the final step we use Proposition \ref{piw} to prove that $\ba_\lambda$ is the polynomial algebra in the (non--zero) images of the elements $P_{i,r}$.


\subsection{} Using the Poincar\' e--Birkhoff--Witt theorem, we may write $$\bu(\CG)=\bu(\CH)\bigoplus \left( \lie C\lie n^-\bu(\CG) +\bu(\CG)\lie C\lie n^+\right),$$ and  let $\bop: \bu(\CG)\to \bu(\CH)$ be the corresponding projection. For $r\in\bz_+$, and $y\in \CG$, let us use the notation
$$
y^{(r)}=\frac{y^r}{r!}\in \bu(\CG).
$$
For $r\in\bz_+$, define
\begin{equation*}
\begin{aligned}
& P_{i,r}= (-1)^r\bop\left(x_{\alpha_i+\delta}^{(r)}x_{-\alpha_i}^{(r)}\right), &&1\le i\le n-1,\\
&P_{n,r}=(-1)^r\bop\left(x_{\frac 12(
\alpha_n+\delta)}^{(2r)}x_{-\alpha_n}^{(r)}\right), && i=n.
\end{aligned}
\end{equation*}
Clearly, the grade of $P_{i,r}$ is $2r$ and, moreover, from the defining relations of $W(\lambda)$ we see that 
\begin{equation}\label{pir}  
P_{i,r} w_\lambda=0,\ \ \text{for }1\le i\le n,\ \ r\ge\lambda(\alpha_i^\vee)+1.
\end{equation} 
More generally,  for  $r\in \bz_+$, we define  
\begin{equation*}
\begin{aligned}
& P_{\alpha,r}= (-1)^r\bop\left(x_{\alpha+\delta}^{(r)}x_{-\alpha}^{(r)}\right),&& \text{for } \alpha\in R^+_s, \\
& P_{\alpha,r}= (-1)^r\bop\left(x_{\alpha+2\delta}^{(r)}x_{-\alpha}^{(r)}\right),&& \text{for }  \alpha\in R^+_\ell,\\
& P_{\frac 12\alpha, r}= (-1)^r\bop\left(x_{\frac 12(\alpha+\delta)}^{(2r)}x_{-\alpha}^{(r)}\right),&& \text{for }  \alpha\in R^+_\ell.
\end{aligned}
\end{equation*}
Note that $P_{i,r}= P_{\alpha_i,r}$ if $1\le i\le n-1$, while $P_{n,r}= P_{\frac 12\alpha_n,r}$.


\subsection{} The following result gives a presentation of the enveloping algebra $\bu(\CH_+)$.

\begin{lem} \label{garl} 
For $1\le i\le n$ and $r\in\bn$, we have $$P_{i,r}=-\frac 1r\sum_{s=0}^{r-1}h_{i,(s+1)\delta}P_{i,r-s-1}.$$ 
Furthermore, $\bu(\CH_+)$ is the graded polynomial algebra on the generators $P_{i,r}$ of grade $2r$, $1\leq i\leq n$, $r\in\bn$.
\end{lem} 
\begin{pf}
For $1\le i \le n-1$ the statement is  a result of Garland (see \cite{CP01} for this formulation). For $i=n$ the result is a similar reformulation of \cite[Corollary 5.39]{F-V98}, (see also \cite[Lemma 3.3(iii)]{CFS08})  where the analogous result was established for  the affine Lie algebra of type $A_2^{(2)}$.
\end{pf}
Using Proposition \ref{sp2}, we see that the elements $P_{\alpha, r}$ for $\alpha\in R^+$  and  $P_{\frac 12\alpha, r}$ and $\alpha\in R^+_\ell$ satisfy an analogous recurrence relation.


\subsection{} The next result describes the action of the elements $P_{i,r}$ on a tensor product of certain representations of $\CG$. 
\begin{lem}\label{tp} 
For $1\le s\le k$, let $V_s$ be a representation of $\CG$ and let  $v_s\in V_s$ be  such that $\lie C\lie n^+ v_s=0$.  For all $1\le i\le n$ and $r\in\mathbb{Z}_+$,  $$P_{i,r}(v_1\otimes\cdots\otimes v_k)=\sum P_{i,j_1}v_1\otimes\cdots\otimes P_{i,j_k}v_k,$$ where the sum is over all $j_s\in\bz_+$ with $j_1+\cdots +j_s=r$.  Analogous statements are true for the elements $P_{\alpha, r}$, $\alpha\in R^+$  and $P_{\frac 12\alpha, r}$, $\alpha\in R^+_\ell$.
\end{lem}\label{tphw}
\begin{pf} Let $\Delta$ be the usual comultiplication in the enveloping algebra $\bu(\CG)$. As it is well-known, if $x, y\in\CG$ and $r,s\in\bz_+$, we have
$$
\Delta\left(x^{(r)}y^{(s)}\right)=\sum_{p, m} x^{(p)}y^{(s-m)}\otimes x^{(r-p)}y^{(m)}.
$$
Keeping in mind that 
$$
P_{i,r}(v_1\otimes\cdots\otimes v_k)=\Delta^{k-1}(P_{i,r})(v_1\otimes\cdots\otimes v_k),
$$
where $\Delta^{k}: \bu(\CG)\to \bu(\CG)^{\otimes k}$ is defined as $(\Delta\otimes \id^{\otimes (k-1)})\cdots(\Delta\otimes \id)\Delta$, our claim follows by using  the definition of $P_{i,r}$ and the fact that $\lie C\lie n^+ v_s=0$.
\end{pf}


\subsection{} Let $\lambda\in P_{2n+1}^+$,  $z\in\bc^\times$, and let $v_\lambda$ be a highest weight vector of $V_{2n+1}(\l)$. Using Theorem \ref{charietal} for $\bpi=(\pi_1,\dots, \pi_{2n})$, with $\pi_i(u)=(1-z^{-1}u)^{\lambda(H_i)}$, $1\leq i\leq 2n$, we obtain
\begin{equation}\label{genmu} 
\left(\sum_{r\ge 0} \Psi_{z,1}(P_{i,r})u^r\right) v_\lambda =\left((1-z^{-1}u)^{\lambda(H_i)}(1+z^{-1}u)^{\lambda(H_{2n+1-i})}\right)v_\lambda, \ \ 1\le i\le n.
\end{equation}  
If $\lambda\in P^+$, the above equation reads
\begin{equation} \label{genmualt}
\left(\sum_{r\ge 0} \Psi_{z,1}(P_{i,r})u^r\right) v_\lambda =(1-z^{-1}u)^{\lambda(\alpha_i^\vee)}v_\lambda, \ \ 1\le i\le n.
\end{equation}  

Assume now that  $\lambda\in P^+$, $\bpi=(\pi_1,\dots, \pi_{2n})$, and $\boz=(z_1,\dots, z_{k})$ are as in the hypothesis of Proposition \ref{piw}. It follows from Proposition \ref{piw} and Lemma \ref{tp} that  the action of $P_{i,r}$ on $w_{\lambda}\in W(\lambda, \bi_\bpi)$ is given by 
\begin{equation} \label{piwpi}
\left( \sum_{r\ge 0}P_{i,r}u^r\right)w_\lambda=\pi_i(u)\pi_{2n+1-i}(-u)w_\lambda, \ \ 1\le i\le n.
\end{equation}
For the above equality, it is important to keep in mind that $\lambda_1,\dots, \lambda_k$ form Proposition \ref{piw} are elements of  $P_{2n+1}^+$ even though $\lambda\in P^+$.


\subsection{} \label{lambdas}  
We are now ready to present the proof of Theorem \ref{alambda}. Let $\tilde \ba_\lambda$ be the quotient of $\bu(\CH_+)$ by the elements  $P_{i,r}$, $1\le i\le n$, $r\ge\lambda(\alpha_i^\vee)+1$. Clearly $\tilde\ba_\lambda$ is a graded algebra  isomorphic to the graded polynomial subalgebra of $\bu(\CH_+)$ generated by the elements $P_{i,r}$, $1\le i\le \lambda(\alpha_i^\vee)$ and  we shall routinely use this. By \eqref{pir} we see that $\ba_\lambda$ is a quotient of $\tilde\ba_\lambda$. It remains to prove that they are isomorphic. We first show that $P_{i,r}$ for $1\le i\le n$, $1\le r\le\lambda(\alpha_i^\vee)$ can act on some quotient of $W(\l)$ in any prescribed fashion.

Let $a_{i,r}\in\bc$ for $1\le i\le n$, $1\le r\le\lambda(\alpha_i^\vee)$. Set $\pi_i(u)=1+\sum_{r=1}^{\lambda(\alpha_i^\vee)}a_{i,r}u^r$ and let $z_1,\dots,z_k$ be the distinct roots of $\pi_1\cdots\pi_n$. Set 
\begin{equation}
\lambda_s=\sum_{i=1}^n m_{i,s}\omega_i,\quad 1\leq s\leq k, \quad \lambda_0=\lambda-\sum_{s=1}^k\lambda_s,
\end{equation} 
where $m_{i,s}\in\mathbb{Z}_+$ is the multiplicity of $z_s$ in $\pi_i$. Equivalently, 
$$
\pi_i(u)=\prod_{s=1}^k (1-z_s^{-1}u)^{\lambda_s(\alpha_i^\vee)}, \quad 1\leq i\leq n.
$$
Consider the $\CG$--module, 
$$
V:=\ev^*_0 V(\lambda_0)\otimes \Psi^*_{z_1,1}V_{2n+1}(\lambda_1)\otimes\cdots\otimes \Psi^*_{z_k,1} V_{2n+1}(\lambda_k),
$$ and denote by $v_{\lambda_s}$ a highest weight vector of $V_{2n+1}(\lambda_s)$, $0\leq s\leq k$. It is clear that  $$\ev_0 (P_{i,r})v_{\lambda_0}=0,$$ for $r\in\bn$, and by equation \eqref{genmualt} we have 
$$\sum_{r\geq 0}^{\lambda_s(\alpha_i^\vee)}\Psi_{z_s,1}( P_{i,r})u^r\cdot v_{\lambda_s}=(1-z_s^{-1}u)^{\lambda_s(\alpha_i^\vee)}v_{\lambda_s}.$$
From Lemma \ref{tp} we obtain that 
$$
\sum_{r\geq 0} P_{i,r}u^r\cdot v_{\lambda_0}\otimes v_{\lambda_1}\otimes \cdots\otimes v_{\lambda_k}= \pi_i(u)v_{\lambda_0}\otimes v_{\lambda_1}\otimes \cdots\otimes v_{\lambda_k},
$$
or, equivalently, 
$$
P_{i,r}\cdot v_{\lambda_0}\otimes v_{\lambda_1}\otimes \cdots\otimes v_{\lambda_k}=a_{i,r}v_{\lambda_0}\otimes v_{\lambda_1}\otimes \cdots\otimes v_{\lambda_k}, \quad 1\le i\le n,~ 1\le r\le\lambda(\alpha_i^\vee).
$$
Since
$$\lie C\lie n^+\cdot v_{\lambda_0}\otimes v_{\lambda_1}\otimes \cdots\otimes v_{\lambda_k}=0,\ \ h\cdot v_{\lambda_0}\otimes v_{\lambda_1}\otimes \cdots\otimes v_{\lambda_k}=\lambda(h)v_{\lambda_0}\otimes v_{\lambda_1}\otimes \cdots\otimes v_{\lambda_k},$$ for all $h\in\lie h$,  and 
$$
x_{-\a_i}^{\lambda(\alpha_i^\vee)+1}\cdot v_{\lambda_0}\otimes v_{\lambda_1}\otimes \cdots\otimes v_{\lambda_k}=0, \quad 1\leq i\leq n,
$$
we conclude that $V$ is a quotient of $W(\lambda)$  with $w_\l$ being mapped to $v_{\lambda_0}\otimes v_{\lambda_1}\otimes \cdots\otimes v_{\lambda_k}$.

Let $\bof\in \tilde\ba_\lambda$, a non-zero element. There is maximal ideal that does not contain $\bof$, therefore there exist $a_{i,r}\in\bc$ for $1\le i\le n$, $1\le r\le\lambda(\alpha_i^\vee)$ such that under the evaluation map sending $P_{i,r}$ to $a_{i,r}$, $1\le i\le n$, $1\le r\le\lambda(\alpha_i^\vee)$, $\bof$ is mapped to a non-zero scalar $\gamma$. Applying the above construction for the scalars $a_{i,r}$, we obtain that there exist a quotient of $W(\lambda)$ on which $\bof$ acts on a non-zero vector by scaling with $\gamma$. We conclude that $\bof$ acts non-trivially on $w_\l$, or, equivalently, $\bof \notin\ann w_\lambda$.   In other words the image of $\bof $ in $\ba_\lambda$ is non-zero and  $\tilde{\ba}_\lambda\cong\ba_\lambda$. The proof of Theorem \ref{alambda} is complete.


\section{Global and Local Weyl modules } \label{six}

\subsection{}
The goal of this section is to show  that the global Weyl module $W(\lambda)$ is a finitely generated $\ba_\lambda$-module. Together with Theorem \ref{alambda} this proves that any local Weyl module $W(\lambda, \bi)$ is finite--dimensional and we obtain a lower bound for this dimension.


\subsection{}

We start by recording some useful relations in $\bu(\CG)$.
\begin{lem}\label{relations}
Let $r\in \bz_+$. Then,
\begin{enumerit}
\item[(i)] For $\alpha \in R^+_s$ , we have 
\begin{equation}\label{1}x_{\alpha+\delta}^{(r)}x_{-\alpha}^{(r+1)}-\left(
x_{-\alpha+r\delta}+\sum^{r}_{s=1}x_{-\alpha+(r-s)\delta}P_{\alpha,s}\right) \in\bu(\CG)\lie C\lie n^+.\end{equation}
\item[(ii)] For $\alpha \in R^+_\ell$, we have
\begin{equation}\label{3}x_{\alpha+2\delta}^{(r)}x_{-\alpha}^{(r+1)}-\left(
x_{-\alpha+2r\delta}+\sum^{r}_{s=1}x_{-\alpha+2(r-s)\delta}P_{\alpha,s}\right) \in\bu(\CG)\lie C\lie n^+,\end{equation}
and
\begin{equation}\label{2}x_{\frac 12 (\alpha+\delta)}^{(2r+1)}x_{-\alpha}^{(r+1)}-\left( x_{\frac
12(-\alpha+(2r+1)\delta)}+\sum^{r}_{s=1}x_{\frac12(-\alpha+(2(r-s)+1)\delta)}P_{\frac 12\alpha,s}\right) \in \bu(\CG)\lie
C\lie n^+.\end{equation}
\end{enumerit}
\end{lem}

\begin{pf}
Part (i) and the first statement in part (ii)   is  deduced from Proposition~\ref{sp2}(ii) and
\cite{G78} and a more detailed proof can be found in \cite{CP01}. The proof of the second assertion in (ii) is a similar  combination of Proposition~\ref{sp2}(iii) and \cite[
Lemma 5.36]{F-V98}.
\end{pf}


\subsection{} We are now ready to prove one of the main results of this section.
\begin{prop}\label{fin-gen} Let $\lambda\in P^+$. The global Weyl module $W(\lambda)$ is a finitely generated $\ba_\lambda$--module. In particular, if $\bi$ is a maximal ideal in $\ba_\lambda$, then the local Weyl module $W(\lambda, \bi)$ is finite--dimensional.
\end{prop}
\begin{pf}  Notice  that  $\wt W(\lambda)$ is a finite set since it is a $W$--invariant subset of $ \lambda-\bar Q^+$. This means that there exists an integer $M\in\bn$ such that  $W(\lambda)$ is generated as a right $\ba_\lambda$--module by the element $w_\lambda$ and elements of the set $\{x_{\beta_m}\cdots x_{\beta_1} w_\lambda: \beta_j\in \widehat R_{\rm{re}}(-),\  1\le m\le M\}$. 

  Let $$\widehat R_{\rm{re}}(\lambda)=\{\beta\in \widehat R_{\rm{re}}(-)~:~ [d, x_\beta]\le N(\lambda)\},\qquad  N(\lambda)=\max\{4\lambda(\alpha^{\vee})+1~:~ \alpha\in R^+\}.$$ Since $\widehat R_{\rm{re}}(\lambda)$ is a finite set,  our claim follows if we prove that $W(\lambda)$ is generated by $w_\lambda$ and the elements of the set 
$$S=\{x_{\beta_m}\cdots x_{\beta_1} w_\lambda: \beta_j\in \widehat R_{\rm{re}}(\lambda),\  1\le m\le M\}.$$

We will show that $x_{\beta_m}\cdots x_{\beta_1} w_\lambda\in S$, for any $\beta_j\in \widehat R_{\rm{re}}(-),\  1\le m\le M$. We proceed by induction on $m$. Let  $m=1$ and $\beta \in\widehat R_{\rm{re}}(-)$, say $\beta=-\alpha+ s\delta$ with $s\in\bz_+$. Using  Lemma \ref{relations}(i,ii) along with the fact that  $x^{\lambda(\alpha^{\vee})+1}_{-\alpha}w_{\lambda}=0$ we  see that $x_{-\alpha+\lambda(\alpha_i^\vee)\delta}w_\lambda=0$ if $\alpha\in R_s^+$ and $x_{-\alpha+2\lambda(\alpha_i^\vee)\delta}w_\lambda=0$ if $\alpha\in R_\ell^+$. Therefore, $x_{-\alpha+s\delta}w_\lambda$ is in the $\ba_\lambda$--module generated by elements $x_{-\alpha+k\delta}w_\lambda$ with $0\le  k\leq 2\lambda(\alpha^\vee)$. By using Lemma \ref{relations}(ii), a similar argument shows that if $\beta=-\frac 12\alpha+(s+\frac 12)\delta$ then $x_{-\frac 12\alpha+(s+\frac 12)\delta}w_\lambda$ is the $\ba_\lambda$--module generated by elements $x_{-\frac 12\alpha+(k+\frac 12)\delta}w_\lambda$ with $0\le  k\leq \lambda(\alpha^\vee)$.

Let $m>1$ and assume that we have proved the assertion for monomials of length at most $m-1$. Consider an element $x_{\beta_m}\cdots x_{\beta_1} w_\lambda$. By the inductive hypothesis, we can assume that $\beta_j\in \widehat R_{\rm{re}}(\lambda)$ for $1\le j\le m-1$. Writing,
$$x_{\beta_m}\cdots x_{\beta_1} w_\lambda=x_{\beta_{m-1}} x_{\beta_m} x_{\beta_{m-2}}\cdots x_{\beta_1} w_\lambda + [x_{\beta_m}, x_{\beta_{m-1}}]x_{\beta_{m-2}}\cdots x_{\beta_1} w_\lambda,$$
and using the induction hypothesis on the two terms on the right completes the proof.
\end{pf}


\subsection{} We shall now study the local Weyl modules in more detail. Let  $\lambda\in P^+$ and let $\bi$ be a maximal ideal in $\ba_\lambda$. By Theorem \ref{alambda} we know that any maximal ideal $\bi$ in $\ba_\lambda$ is  generated by elements $P_{i,r}-a_{i,r}$ for some $a_{i,r}\in\bc$ for $1\le i\le n$, $\ 1\le r\le\lambda(\alpha_i^\vee)$. The augmentation ideal $\bi_{\lambda,0}$  corresponds to taking $a_{i,r}=0$ for all $i,r$. 

Set 
$$
\tilde{\pi_i}(u)=1+\sum_{r\ge 1}^{\lambda(\alpha_i^\vee)}a_{i,r}u^r, \quad 1\le i\le n.
$$  
Let $z_1,\dots, z_k$ be a set of representatives for the set of distinct roots of $\tilde{\pi}_1\cdots\tilde{\pi}_n$ up to sign. Let $r_{i,s},~m_{i,s}\in \mathbb{Z}_+$ be the multiplicity   in $\tilde{\pi}_i$ of $z_s$ and $-z_s$, respectively.  Define $\bpi=(\pi_1,\dots, \pi_{2n})$ by, 
\begin{equation*}
\begin{aligned}
&{\pi}_i(u)=\prod_{s=1}^k(1-z_s^{-1}u)^{r_{i,s}}, &&1\le i\le n, \\
&{\pi}_i(u)= \prod_{s=1}^k(1-z_s^{-1}u)^{m_{2n+1-i,s}}, && n+1\le i\le 2n.
\end{aligned}
\end{equation*}
Denote $\boz=(z_1,\dots, z_k)$ and $\mu=\lambda-\sum_{i=1}^n \sum^k_{s=1}(r_{i,s}+m_{i,s})\omega_i$.

\begin{prop}\label{weylsl} 
With the notation above, for a sufficiently large integer  $N$, there is a surjective morphism of $\CG$--modules $$W(\lambda,\bi)\to W(\mu,\bi_{\mu,0})\otimes \Psi^*_{\boz, N}  W_{2n+1}(\bpi).$$
Furthermore,
$$
\dim W(\lambda,\bi)\ge \dim W(\mu,\bi_{\mu,0}) \prod_{i=1}^{2n}\binom{2n+1}{i}^{(\lambda-\mu)(\alpha_i^\vee)}.
$$

\end{prop}
\begin{pf} By Proposition \ref{piw}, we see that  $\Psi^*_{\boz, N}  W_{2n+1}(\bpi)$ is generated as a $\CG$--module by $w_\bpi$. In fact, more generally, by using Remark \ref{pullbacks}, we see that for any $r\in\bz_+$, we have $$ \Psi^*_{\boz, N}  W_{2n+1}(\bpi)= \bu\left(\oplus_{s\ge r}\CG[s]\right) w_\bpi.$$ Since $W(\mu, \bi_{\mu,0})$ is finite dimensional graded $\CG$--module it follows that 
$$\CG[N]W(\mu, \bi_{\mu,0})= 0,
$$
for $N$ sufficiently large. Hence,   $W(\mu,\bi_{\mu,0})\otimes \Psi^*_{\boz, N}  W_{2n+1}(\bpi)$ is generated as a $\CG$--module by the element  $w= w_\mu\otimes w_\bpi$. Our first claim follows with the help of equation \eqref{piwpi} by noting that $w$  satisfies the defining relations of $W(\lambda, \bi)$. The second claim now follows from Proposition \ref{piw}.
\end{pf}


\subsection{} We are now ready to prove an important dimensional inequality for local Weyl modules.
\begin{prop}\label{dimineq} Let $\lambda \in P^+$ and   let $\bi$ be a maximal ideal in $\ba_\lambda$. We have $$\dim W(\lambda,\bi_{\lambda,0})\ge \dim W(\lambda,\bi_{})\ge \prod_{i=1}^n\binom{2n+1}{i}^{\lambda(\alpha_i^\vee)},$$ where $\bi_{\lambda,0}$ is the augmentation ideal of $\ba_\lambda$.
\end{prop}
\begin{pf}
We recall a general construction from \cite{FL99}. Let $\bu(\CG)[k]$ be the homogeneous component of degree $k$  and recall that it is  a $\lie g$--module for all $k\in\mathbb{Z}_+$. Suppose now that we are given a $\CG$--module $V$ which is generated by $v$. Define an increasing filtration $0\subset V^0\subset V^1\subset\cdots $ of $\lie g$-submodules of $V$ by$$V^k=\bigoplus_{s=0}^k \bu(\CG)[s] v.$$ The associated graded vector space  $\gr V$ admits an action of $\CG$ given by: 
$$
x(v+V^{k})= xv+ V^{k+s},\ \ x\in\CG[s],\ \ v\in V^{k+1}.
$$
Moreover, the image of $v$ in  $\gr V$  generates $\gr V$ as a $\CG$--module.

If $\dim V<\infty$ it is clear that $V$ and $\gr V$ are isomorphic  as $\lie g$-modules but in general not as $\CG$-modules.
Applying  this construction to the local Weyl module $W(\lambda,\bi)$ for a maximal ideal $\bi$, we see that in $\gr W(\lambda,\bi)$ we have, in addition to the relations of the global Weyl module, the relation $$P_{i,r}\bar w_{\lambda,\bi}=0$$ where $\bar w_{\lambda,\bi}$ is the image of $w_{\lambda,\bi}$ in $\gr W(\lambda,\bi)$.  Hence $\gr W(\lambda,\bi)$ is a quotient of $W(\lambda,\bi_{\lambda,0})$ and we have proved,
 $$\dim W(\lambda, \bi_{\lambda,0})\ge\dim\gr W(\lambda, \bi)=\dim W(\lambda,\bi).$$

The proof of the remaining inequality is completed by an induction on $\sum_i\lambda(\alpha_i^\vee)$  as follows. By Proposition \ref{weylsl} we have
$$
\dim W(\lambda,\bi)\ge \dim W(\mu,\bi_{\mu,0}) \prod_{i=1}^{2n}\binom{2n+1}{i}^{(\lambda-\mu)(\alpha_i^\vee)}.
$$
for some $\mu\in P^+$ with the property that $\lambda-\mu\in P^+$. If $\mu=0$, then this is the desired result, while if $\mu\ne 0$ our induction hypothesis applies and the inductive step follows.
\end{pf}

\section{Local Weyl modules and Demazure modules}\label{seven}

\subsection{} The goal of this section is to prove Theorem \ref{locWeyl=Dem} and Theorem \ref{mainthm}.

The strategy for approaching Theorem \ref{locWeyl=Dem} is the one from \cite[Theorem 7]{FoL07} which reduces the argument to rank one irreducible affine Lie algebras. The rank one irreducible affine Lie algebras are those of type $A_1^{(1)}$ and $A_{2}^{(2)}$. For the affine Lie algebra of type $A_1^{(1)}$ the statement is known \cite[Corollary 1.5.1]{CL06}; for the affine Lie algebra of type $A_{2}^{(2)}$ the proof is presented in Appendix \ref{a22Dem}.


\subsection{} We start by recalling the definition of the local Weyl module $W(\lambda,\bi_{\lambda,0})$. The local Weyl module $W(\lambda,\bi_{\lambda,0})$ is the $\CG$-module generated by an element $w_{\l, \bi_{\lambda,0}}$ with the relations
\begin{equation}\label{lWeylrels}
\begin{aligned}
&(\CN^+\oplus \CH_+) \cdot w_{\l, \bi_{\lambda,0}}=0,\\
&h\cdot w_{\l, \bi_{\lambda,0}}=\l(h)w_{\l, \bi_{\lambda,0}},~h\in\H,\\
&x_{-\a}^{(\l,\a^\vee)+1}\cdot w_{\l, \bi_{\lambda,0}}=0, ~ \a\in R^+ .
\end{aligned}
\end{equation}
Note that since the ideal $\bi_{\lambda,0}$ is homogeneous the module $W(\lambda,\bi_{\lambda,0})$ is graded.

Similarly, Demazure modules can also be presented  as cyclic modules that have an explicit description of the annihilator of the generating element. The description of the Demazure modules for finite dimensional simple Lie algebras  \cite[Theorem 3.4]{J85}, \cite[Proposition Fondamentale 2.1]{P89} was extended in \cite[Lemme 26]{M89} for Demazure modules associated to Kac-Moody Lie algebras. We record below the statement that is relevant for us. 

\begin{thm}
Let $\l\in P^+$. The Demazure module $D(-\l)$ is the $(\H\oplus \bc d)\oplus \Naff^+$-module generated by an element $e_{-\l}$ with the relations
\begin{equation}\label{b-Demrels}
\begin{aligned}
&x_\a^{k_\a+1} \cdot e_{-\l}=0,~ \a\in \widehat{R}_{\rm re}^+, ~k_\a=\max\{0,~-(-\l+\L_0,\a^\vee)\},\\
&\CH_+ \cdot e_{-\l}=0,\\
&h\cdot e_{-\l}=-\l(h)e_{-\l},~h\in\H\oplus \bc d.\\
\end{aligned}
\end{equation}
\end{thm}

\subsection{}

 Note that the Demazure modules for $\Gaff$ are $\Baff$-modules; the description above is that of $D(-\l)$ as a $(\H\oplus \bc d)\oplus \Naff^+$-module. Since $-\l$ is anti-dominant, $D(-\l)$ is admits a structure of  graded $\CG$-module such that $e_{-\l}$ becomes a lowest-weight vector. Let us denote by $e_\l$ the (unique up to scaling) highest weight vector in the $\G$-submodule of $D(-\l)$ spanned by $e_{-\l}$.
 
\begin{cor}
Let $\l\in P^+$. The Demazure module $D(-\l)$ is the graded $\CG$-module generated by a degree zero element $e_{\l}$ with the relations
\begin{equation}\label{k-Demrels-final}
\begin{aligned}
&x_\a^{k_\a+1} \cdot e_{\l}=0,~ \a\in \widehat{R}_{\rm re}^+\cup R^-, ~k_\a=\max\{0,~-(\l+\L_0,\a^\vee)\},\\
&\CH_+ \cdot e_{\l}=0,\\
&h\cdot e_{\l}=\l(h)e_{\l},~h\in\H.\\
\end{aligned}
\end{equation}
\end{cor}

\subsection{}

For $\a\in \widehat{R}_{\rm re}(+)$, remark that $(\l+\L_0,\a^\vee)\geq 0$, therefore the combining the first relation in \eqref{k-Demrels-final} for this type of roots with the second relation in \eqref{k-Demrels-final} we can write
$$
(\CN^+\oplus\CH_+) \cdot e_{\l}=0.
$$
For $\a\in R^{+}$, remark that $-(\l+\L_0,-\a^\vee)=(\l,\a^\vee)\geq 0$. Therefore, the first relation in \eqref{k-Demrels-final} for the roots $-\a$ with $\a\in R^+$ can be written as
$$
x_{-\a}^{(\l,\a^\vee)+1}\cdot e_{\l}=0, ~ \a\in R^+ .
$$
Therefore, we can state the following.
\begin{cor}\label{annihilators}
Let $\l\in P^+$. With the notation above, $\Ann_{\bu(\CG)}(e_\l)$ is the left ideal of $\bu(\CG)$ generated by $\Ann_{\bu(\CG)}(w_{\l, \bi_{\lambda,0}})$ and the elements
\begin{equation}
x_\a^{k_\a+1},~ \a\in \widehat{R}_{\rm re}(-)\setminus R^-, ~k_\a=\max\{0,~-(\l+\L_0,\a^\vee)\}.
\end{equation}
\end{cor}

Recall that the set  $\widehat{R}_{\rm re}(-)\setminus R^-$ consists of the following affine roots
\begin{equation*}
\widehat{R}_{\rm re}(-)\setminus R^-=\frac{1}{2}(R^-_\ell+(2\bz_{+}+1)\d)\cup(R^-_s+\bn\d)\cup(R^-_\ell +2\bn\d).
\end{equation*}

\subsection{}
The following result will proved by reduction to rank one.

\begin{prop}\label{higherrank}
Let $\l\in P^+$. Then,
$$
x_\a^{k_\a+1}\in \Ann_{\bu(\CG)}(w_{\l, \bi_{\lambda,0}}), \quad \text{for all }~ \a\in \widehat{R}_{\rm re}(-)\setminus R^-, ~k_\a=\max\{0,~-(\l+\L_0,\a^\vee)\}.
$$
\end{prop}
\begin{proof} 
For the affine Lie algebra of type $A_1^{(1)}$ the corresponding statement is known \cite[Corollary 1.5.1]{CL06}. For the affine Lie algebra of type $A_2^{(2)}$, the statement is Proposition \ref{rankoneA22}.  

We may assume that $\Gaff$ has rank at least two. Let $\a\in R^-$. By Proposition \ref{sp2}(ii), there exists an idecomposable affine Lie subalgebra of $\Gaff$  of type $A_1^{(1)}$ whose non-imaginary root spaces are precisely the root spaces of $\Gaff$ parametrized by 
$$
\widehat{R}_{\rm re, \a}:=(\pm\a+\bz\d)\cap \widehat{R}.
$$
Let us denote the corresponding current algebra by $\CG_\a$, by  $(\cdot,\cdot)_\a$ the standard bilinear form, and by $\L_{\a,0}, \L_\a$ the fundamental weights. An element $\beta\in \widehat{R}_{\rm re,\a}$ is also an element of $\widehat{R}_{\rm re}$. When referring to the co-root corresponding to $\beta$ as an element of $\widehat{R}_{\rm re}$ we will use $\beta^\vee$ as usual, and when referring to the co-root corresponding to $\beta$ as an element of $\widehat{R}_{\rm re,\a}$ we will use $\beta^{\vee_\a}$. Note that with respect to $(\cdot,\cdot)_\a$, the root $\a$ has always square length 2.  If $\a\in R^-_s$ the null-root $\d_\a$ equals $\d$ and the $0$-fundamental weight $\L_{\a,0}$ equals $\L_0$.  If $\a\in R^-_\ell$ the null-root $\d_\a$ equals $2\d$ and the $0$-fundamental weight $\L_{\a,0}$ equals $\frac{1}{2}\L_0$. 

Denote $\l_\a:=-(\l,\a^\vee)\L_\a$.  We have
\begin{equation}\label{ks}
(\l_\a+\L_{\a,0}, (\a+s\d_\a)^{\vee_\a} )_\a=(\l,\a^\vee)+s=(\l+\L_0, (\a+s\d)^\vee).
\end{equation}

If $w_{\l_\a, \bi_{\lambda_\a,0}}$ denotes the highest weight cyclic vector of the local Weyl module $W(\lambda_\a,\bi_{\lambda_\a,0})$ for $\CG_\a$ with highest weight $\l_\a$, then the rank one statement  for $\CG_\a$ and $\l_\a$, together with \eqref{ks} imply  that 
$$
x_\beta^{k_\beta+1}\in \Ann_{\bu(\CG_\a)}(w_{\l_\a, \bi_{\lambda_\a,0}}), \quad \text{for all }~ \beta\in (\a+\bn\d)\cap \widehat{R}.
$$
Noting that according to \eqref{lWeylrels} we have
$$
\Ann_{\bu(\CG_\a)}(w_{\l_\a, \bi_{\lambda_\a,0}})\subseteq \Ann_{\bu(\CG)}(w_{\l, \bi_{\lambda,0}}), 
$$
we obtain that $x_\beta^{k_\beta+1}\in \Ann_{\bu(\CG)}(w_{\l}, \bi_{\lambda,0})$ for all $\beta\in (\a+\bn\d)\cap \widehat{R}$. Hence, our statement is proved for all elements of the following set
\begin{equation*}
(R^-_s+\bn\d)\cup(R^-_\ell +2\bn\d)\quad \text{if }  n\geq 2.
\end{equation*}
We are left with showing that our statement holds for the roots in
$$
\frac{1}{2}(R^-_\ell +(2\bz_{+}+1)\d)\quad \text{if }  n\geq 2.
$$
Let $\a\in R^-_\ell$. By Proposition \ref{sp2}(iii), there exists an idecomposable affine Lie subalgebra of $\Gaff$  of type $A_2^{(2)}$ whose non-imaginary root spaces are precisely the root spaces of $\Gaff$ parametrized by 
$$
\widehat{R}_{\rm re,\a}:=\frac{1}{2}(\pm\a+(2\bz+1)\d)\cup (\pm\a+2\bz\d).
$$
As before, let us denote the corresponding algebra by $\CG_\a$, by  $(\cdot,\cdot)_\a$ the standard bilinear form, and by $\L_{\a,0}, \L_\a$ the fundamental weights. 
Note that in this case, two standard bilinear forms coincide on the common domain, as do the co-roots, null-roots, and $0$-fundamental weights. Denote $\l_\a:=-(\l,\a^\vee)\L_\a$.  Then,
$$
(\l_\a+\L_{0}, \beta^{\vee} )=(\l+\L_0, \beta^\vee), \quad \text{for all }~\beta\in \frac{1}{2}(\a+(2\bz+1)\d)\cup (\a+2\bz\d).
$$
If $w_{\l_\a, \bi_{\lambda_\a,0}}$ denotes the highest weight cyclic vector of the local Weyl module $W(\lambda_\a,\bi_{\lambda_\a,0})$ for $\CG_\a$ with highest weight $\l_\a$, then Proposition \ref{rankoneA22} for $\CG_\a$ and $\l_\a$,  imply  that 
$$
x_\beta^{k_\beta+1}\in \Ann_{\bu(\CG_\a)}(w_{\l_\a, \bi_{\lambda_\a,0}}), \quad \text{for all }~ \beta\in \frac{1}{2}(\a+(2\bz_++1)\d)\cup (\a+2\bn\d).
$$
As before, from \eqref{lWeylrels} we obtain
$$
\Ann_{\bu(\CG_\a)}(w_{\l_\a,, \bi_{\lambda_\a,0}})\subseteq \Ann_{\bu(\CG)}(w_{\l, \bi_{\lambda,0}}), 
$$
and therefore, $x_\beta^{k_\beta+1}\in \Ann_{\bu(\CG)}(w_{\l, \bi_{\lambda,0}})$ for all $\beta\in \frac{1}{2}(\a+(2\bz_++1)\d)\cup (\a+2\bn\d)$. To conclude, our statement is proved for all the elements of the set
$$
\frac{1}{2}(R^-_\ell +(2\bz_{+}+1)\d),
$$
and hence it holds for all  roots in $\widehat{R}_{\rm re}(-)\setminus R^-$.
\end{proof}

\subsection{} We shall now present the proof of Theorem \ref{locWeyl=Dem}.  Proposition \ref{higherrank} and Corrollary \ref{annihilators} imply that both $W(\lambda,\bi_{\lambda,0})$ and $D(-\l)$ are graded cyclic $\CG$-modules and the cyclic vectors have the same annihilator ideal. Therefore, $W(\lambda,\bi_{\lambda,0})$ and  $D(-\l)$ are isomorphic as graded $\CG$-modules and, by restriction, they are also isomorphic as graded $\G$-modules.

\subsection{} Let us recall \cite[Theorem 2, pg. 195]{FoL06}.
\begin{prop}\label{Demtensor}
Let $\lambda \in P^+$. Then $D(-\l)$ and $\otimes_{i=1}^nD(-\L_i)^{\otimes \lambda(\alpha_i^\vee)}$ are isomorphic as graded $\G$-modules. In particular, $$\dim D(-\l)=\prod_{i=1}^n\dim D(-\L_i)^{\lambda(\alpha_i^\vee)}.$$
Furthermore, as $\G$-modules,
$$
D(-\L_i)\cong V(0)\oplus V(\L_1)\oplus\cdots\oplus V(\L_i).
$$
\end{prop}

The Weyl dimension formula can be used to obtain
$$
\dim V(\L_i)=\frac{(2n-2i+2)\cdot (2n+1)!}{i!\cdot (2n-i+2)!}=\binom{2n+1}{i}-\binom{2n+1}{i-1}.
$$

\subsection{} We present the proof of Theorem \ref{mainthm}. Let $\lambda\in P^+$  and  let $\bi$ be a maximal ideal in $\ba_\lambda$.  From Theorem \ref{locWeyl=Dem}, Proposition \ref{Demtensor}, and the above dimension computation we obtain that
$$
\dim W(\lambda,\bi_{\lambda,0})=\prod_{i=1}^n\binom{2n+1}{i}^{\lambda(\alpha_i^\vee)}.
$$
This, together with Proposition \ref{dimineq} implies the conclusion of Theorem \ref{mainthm}.



\section{Local Weyl modules and Demazure modules in rank one} \label{a22Dem}


\subsection{}  We shall present the proof of Proposition \ref{higherrank} and Theorem \ref{locWeyl=Dem} for $\Gaff$ of type $A_2^{(2)}$. This proof emulates the proof of the same statement for $\Gaff$ of type $A_1^{(1)}$ given in \cite{CV13}.

Consider the current algebra $\CG$ of type $A_{2}^{(2)}$. In this case, $R^{+}=\{\a\}$, the unique fundamental weight is $\L=\a/2$, and 
$$\widehat{R}_{re}(\pm)=(\pm\a+2\bz_{+}\d)\cup (\pm\frac{\a}{2}+(\frac{1}{2}+\bz_{+})\d).$$
We can fix a basis $\{x_\beta, ~h_{s\d}~:~\beta\in \widehat{R}_{re}(\pm), s\in\bz_+\}$ for $\CG$ as in Proposition \ref{basis}. We recall here the relation the relations in $\bu(\CG)$ that will be used later on. We first introduce some notation.

For any $\ell,m\in\frac{1}{2}\bz_{+}$ such that $\ell+m\in \bz$, let $S(\ell,m)$ be the set of non-negative integer sequences $(p_i)_{i\in \frac{1}{2}\bz_+}$ that satisfy
\begin{equation}\label{a22sequences}
\begin{aligned}
&\sum_{N\geq 0} \frac{1}{2}p_{N+\frac{1}{2}} +\sum_{N\geq 0} p_N=\ell,\\
&\sum_{N\geq 0} \frac{2N+1}{2}p_{N+\frac{1}{2}} +\sum_{N\geq 0} 2Np_N=m.
\end{aligned}
\end{equation}
We define the support of $p\in S(r,s)$ as 
$$
\sup(p)=\{i\in \frac{1}{2}\bz_+~|~p_i\neq 0\}.
$$
Also, let 
\begin{equation}\label{ydef}
y(\ell,m)=\sum_{p\in S(\ell,m)} \overset{\rightarrow}{\prod}_{N\geq 0}
\left(
\frac{(-1)^{p_{N+\frac{1}{2}}}}{2^{Np_{N+\frac{1}{2}}}}
{x_{-\frac{\a}{2}+(N+\frac{1}{2})\d}^{(p_{N+\frac{1}{2}})}}
\right)\left(
\frac{(-1)^{p_{N}}(2-(-1)^N)^{p_N}}{2^{2Np_{N}}}
{x_{-\a+2N\d}^{(p_N)}}
\right).
\end{equation}
Above, $\overset{\rightarrow}{\prod}_{N\geq 0}$ refers to the product of the specified factors written exactly in the increasing order of the indexing parameter (the factors do not commute and the order in which they appear in the product is important).

\begin{lem}\label{fullrelations}
Let $\ell,m\in\frac{1}{2}\bz_{+}$ such that $\ell+m\in \bz$. Then,
$$
(-1)^{\ell+m}{x_{\frac{\a}{2}+\frac{1}{2}\d}^{(2m)}}{x_{-\a}^{(\ell+m)}}-y(\ell,m)\in \bu(\CG)\CN^+ .
$$
\end{lem}

We refer to \cite[Corollary 5.39]{F-V98}  for a proof.

\subsection{} For $k,\ell,m\in\frac{1}{2}\bz_+$ such that $\ell+m\in \bz$, let 
$$
S_{<k}(\ell,m)=\{p\in S(\ell,m)~|~\sup(p)\subseteq [0,k) \}$$ and $$S_{\geq k}(\ell,m)=\{p\in S(\ell,m)~|~\sup(p)\subseteq [k,\infty) \}.
$$
Consider $y_{<k}(\ell,m)$ and $y_{\geq k}(\ell,m)$ defined as in \eqref{ydef} with the index set for the summation is $S_{<k}(\ell,m)$ and respectively $S_{\geq k}(\ell,m)$.

With this notation, we have

\begin{equation}\label{eq4}
y(\ell, m)=y_{\geq k+\frac{1}{2}}(\ell,m)+\sum_{(i,r)\in T(\ell,m,k)}  y_{<k+\frac{1}{2}}(\ell-i,m-r) y_{\geq k+\frac{1}{2}}(i,r) ,
\end{equation}
where 
\begin{equation}
T(\ell,m,k)=\{(i,r)~|~i,r\in\frac{1}{2}\bz_+,~ i+r\in\bz,~i<\ell,~0\leq m-r\leq 2k(\ell-i), ~ r\geq (2k+1)i\}.
\end{equation}
The constraint that $m-r\leq 2k(\ell-i)$ arises from the fact that if  $m-r> 2k(\ell-i)$ then $S_{< k+\frac{1}{2}}(i,r)=\emptyset$ (in which case $y_{< k+\frac{1}{2}}(i,r)=0$). The constraint that $r\geq (2k+1)i$ arises from the fact that if  $r< (2k+1)i$ then $S_{\geq k+\frac{1}{2}}(i,r)=\emptyset$ (in which case $y_{\geq k+\frac{1}{2}}(i,r)=0$). Both verifications are straightforward from \eqref{a22sequences}.
\subsection{} Let us consider the local Weyl module $W(n\L,\bi_{n\L,0})$, $n\geq 0$, for $\CG$. It is generated by the highest weight vector $w_n$ with the relations
\begin{equation}
\begin{aligned}
& (\CN^+\oplus \CH_+)\cdot w_n=0,\\
& h_0\cdot w_n=nw_n,\\
& x_{-\a}^{n+1}\cdot w_n=0.
\end{aligned}
\end{equation}
Our goal is to show that the following relations are also satisfied
\begin{equation}\label{sl2rels}
\begin{aligned}
&x_{-\a+2k\d}^{n-k+1}\cdot w_n=0, \text{ for } 0\leq k\leq n-1,\\
&x_{-\a+2k\d}\cdot w_n=0, \text{ for } k\geq  n,
\end{aligned}
\end{equation}
and
\begin{equation}\label{a22rels}
\begin{aligned}
&x_{-\frac{\a}{2}+(k+\frac{1}{2})\d}^{2n-2k}\cdot w_n=0, \text{ for } 0\leq k\leq n-1,\\
&x_{-\frac{\a}{2}+(k+\frac{1}{2})\d}\cdot w_n=0, \text{ for } k\geq  n,
\end{aligned}
\end{equation}
The relations \eqref{sl2rels} follow from the fact that $\oplus_{n\in\bz}\Gaff_{\pm\a+2n\d}\oplus\H$ is the current algebra of type $A_1^{(1)}$. We will focus on the relations \eqref{a22rels}. Our starting point is the following.
\begin{lem}\label{Lm1}
For $\ell,m\in\frac{1}{2}\bz_+$ such that $\ell+m\in \bz$ we have
\begin{equation}
y(\ell,m)\cdot w_n
=
(-1)^{\ell+m}{x_{\frac{\a}{2}+\frac{1}{2}\d}^{(2m)}}{x_{-\a}^{(\ell+m)}}\cdot w_n
\end{equation}
In particular, if $\ell+m\geq n+1$, we have
\begin{equation}\label{eq1}
y(\ell,m)\cdot w_n=0.
\end{equation}
\end{lem}
\begin{pf}
Straightforward from Lemma \ref{fullrelations}.
\end{pf}
\subsection{} We can now show that the second relation in \eqref{a22rels} holds.
\begin{lem}\label{Lm2}
For $k\in\bz_+$ such that $k\geq n$ we have
\begin{equation}
x_{-\frac{\a}{2}+(k+\frac{1}{2})\d}\cdot w_n=0.
\end{equation}
\end{lem}
\begin{proof}
By \eqref{eq1} we have
$$
y(\frac{1}{2},k+\frac{1}{2})\cdot w_n=0.
$$
Now, the unique element in $S(\frac{1}{2},k+\frac{1}{2})$ is the sequence $p$ for which $p_i=\d_{i,k+\frac{1}{2}}$, $i\in\frac{1}{2}\bz_+$ and hence
$$y(\frac{1}{2},k+\frac{1}{2})=-\frac{1}{2^{k}}
x_{-\frac{\a}{2}+(k+\frac{1}{2})\d},$$
from which the desired conclusion follows.
\end{proof}
\subsection{} The overall argument rest on the following two results.

\begin{lem}\label{Lm3}
Let $k\in\bz_+$  and $m\in \frac{1}{2}+\bz_+$ such that $m+\frac{1}{2}>n$ and $m\geq k+\frac{1}{2}$. Then,
$$
y(\frac{1}{2},m)=y_{\geq k+\frac{1}{2}}(\frac{1}{2},m)
$$
and $y_{\geq k+\frac{1}{2}}(\frac{1}{2},m)\cdot w_n=0$.
\end{lem}
\begin{proof}
The unique element in $S(\frac{1}{2},m)$ is the sequence $p$ for which $p_i=\d_{i,m}$, $i\in\frac{1}{2}\bz_+$. Since $m\geq k+\frac{1}{2}$ this sequence belongs to 
$S_{\geq k+\frac{1}{2}}(\frac{1}{2},m)$ and therefore
$$
y(\frac{1}{2},m)=y_{\geq k+\frac{1}{2}}(\frac{1}{2},m).
$$
The second claim then follows from \eqref{eq1}.
\end{proof}

\subsection{} The following is an extension of Lemma \ref{Lm3}.

\begin{lem}\label{Lm4}
Let $k\in\bz_+$. If $\ell,m\in\frac{1}{2}\bz_+$ such that 
\begin{equation}\label{eq2}
\ell+m\in \bz, \quad m+\ell>(2k+1)\ell+n-k-\frac{1}{2},\quad \text{and}\quad m\geq (2k+1)\ell.
\end{equation}
Then, 
\begin{equation}\label{eq3}
y_{\geq k+\frac{1}{2}}(\ell,m)\cdot w_n=0.
\end{equation}
\end{lem}
\begin{proof}
We will prove the statement by induction on $\ell\in \frac{1}{2}\bz_+$. Remark first that if $\ell=0$ then $S_{\geq k+\frac{1}{2}}(0,m)=\emptyset$ and therefore $y_{\geq k+\frac{1}{2}}(0,m)=0$, which implies the conclusion.

If $\ell=\frac{1}{2}$, then \eqref{eq2} is equivalent with the fact that $m$ satisfies the hypotheses of Lemma \ref{Lm3}. Therefore, by Lemma \ref{Lm3} our conclusion is satisfied.

Assume now that $\ell\geq 1$ and that \eqref{eq3} holds for all pairs $(i,r)$ satisfying \eqref{eq2} with $i<\ell$. Let $(i,r)\in T(\ell,m,k)$. Then, by the definition of $T(\ell,m,k)$, we have
$$
m+\ell=(m-r)+r+\ell\leq 2k(\ell-i)+r+\ell.
$$
On the other hand, by hypothesis,
$$
(2k+1)\ell+n-k-\frac{1}{2}<m+\ell
$$
so, 
$$
(2k+1)\ell+n-k-\frac{1}{2}<2k(\ell-i)+r+\ell
$$
which is equivalent to 
$$
(2k+1)i+n-k-\frac{1}{2}<r+i.
$$
Now, $r\geq (2k+1)i$ and $i+r\in \bz$ by the definition of $T(\ell,m,k)$, which means that $(i,r)$ satisfies \eqref{eq2} and, according to our induction hypothesis, 
$$
y_{\geq k+\frac{1}{2}}(i,r)\cdot w_n=0
$$
for all $(i,r)\in T(\ell,m,k)$. From \eqref{eq4} we obtain
\begin{equation}\label{eq5}
y(\ell,m)\cdot w_n=y_{\geq k+\frac{1}{2}}(\ell,m)\cdot w_n.
\end{equation}
Remark that, since $\ell\geq 1$, we have
$$m+\ell>(2k+1)\ell+n-k-\frac{1}{2}\geq 2k+1+n-k-\frac{1}{2}=n+k+\frac{1}{2}\geq n+\frac{1}{2},$$
but since $m+\ell\in \bz$ we can write
$$
m+\ell\geq n+1.
$$
In this situation Lemma \ref{Lm1} applies and therefore
$$
y(\ell,m)\cdot w_n=0.
$$
Combining this with \eqref{eq5} we obtain that 
$$
y_{\geq k+\frac{1}{2}}(\ell,m)\cdot w_n=0,
$$
which is our desired conclusion.
\end{proof}

\subsection{} We are now ready to show that the first relation in \eqref{a22rels} holds.

\begin{prop}\label{Prop1}
For $k\in\bz_+$ such that $0\leq k\leq n-1$ we have,
\begin{equation}
x_{-\frac{\a}{2}+(k+\frac{1}{2})\d}^{2n-2k}\cdot w_n=0.
\end{equation}
\end{prop}
\begin{proof}
We apply Lemma \ref{Lm4} for $\ell=n-k$ and $m=(2k+1)(n-k)$ for which the hypotheses of Lemma \ref{Lm4} are easily checked. We obtain that
\begin{equation}\label{eq6}
y_{\geq k+\frac{1}{2}}(n-k,(2k+1)(n-k))\cdot w_n=0.
\end{equation}
Let $p\in S_{\geq k+\frac{1}{2}}(n-k,(2k+1)(n-k))$. By definition, we have 
\begin{equation}
\begin{aligned}
&\sum_{N\geq k} \frac{1}{2}p_{N+\frac{1}{2}} +\sum_{N\geq k+1} p_N=n-k,\\
&\sum_{N\geq k} \frac{2N+1}{2}p_{N+\frac{1}{2}} +\sum_{N\geq k+1} 2Np_N=(2k+1)(n-k) .
\end{aligned}
\end{equation}
Multiplying the first equality by $2k+1$ and subtracting from the second equality we obtain
$$
\sum_{N\geq k} \frac{(2N+1)-(2k+1)}{2}p_{N+\frac{1}{2}} +\sum_{N\geq k+1} (2N-2k-1)p_N=0.
$$ 
All terms in the above equality are non-negative, which implies that they must be all zero. Therefore, 
$$
p_i=(2n-2k)\delta_{i,k+\frac{1}{2}}, ~i\in \frac{1}{2}\bz_+.
$$
As $S_{\geq k+\frac{1}{2}}(n-k,(2k+1)(n-k))$ contains a unique sequence, we have
$$
y_{\geq k+\frac{1}{2}}(n-k,(2k+1)(n-k))=\frac{1}{2^{2k(n-k)}}
{x_{-\frac{\a}{2}+(k+\frac{1}{2})\d}^{(2n-2k)}}.
$$
Together with \eqref{eq6}, this implies our conclusion.
\end{proof}

\subsection{}

We have proved the following.
\begin{prop}\label{rankoneA22}
Let $\Gaff$ be an  affine Lie algebra of type $A_2^{(2)}$, and let $\l\in P^+$. Then,
$$
x_\a^{k_\a+1}\in \Ann_{\bu(\CG)}(w_\l), \quad \text{for all }~ \a\in \widehat{R}_{re}(-)\setminus R^-, ~k_\a=\max\{0,~-(\l+\L_0,\a^\vee)\}.
$$
\end{prop}
Taking into account Corollary \ref{annihilators} this is equivalent to Theorem \ref{locWeyl=Dem} for $\Gaff$  of type $A_2^{(2)}$.


\section{Weyl modules for the standard maximal parabolic subalgebras of twisted affine Lie algebras}  \label{nine}

\subsection{} In this section, we consider the special maximal parabolic subalgebra in the remaining idecomposable twisted affine Lie algebras. We shall be fairly brief here since  the treatment is conceptually identical and technically easier than the considerations for the hyperspecial maximal parabolic of the affine Lie algebra of type $A_{2n}^{(2)}$. The papers \cite{FMSa, FMS} discuss some of these ideas in greater generality but the precise results that we prove here are not stated, their methods are different, and the statements that might be relevant for us are conditional on the validity of certain assumptions. It  seems worthwhile to include a short self-contained treatment in this paper.


\subsection{}  

Let $\G$ be a simple Lie algebra of type $A_{2n-1}$, $n\geq 3$, $D_{n+1}$, $n\geq 2$, or $E_6$. The outer automorphism group of $\G$ is non-trivial. We fix  a non-trivial outer automorphism $\sigma$ and we denote by $m$ its order. The only possible value of $m$ is $2$ with the exception of the Lie algebra of type $D_4$ for which $m$ could also be equal to $3$. All the outer automorphisms arise from an automorphism of the Dynkin diagram of $\G$ and we can therefore regard $\sigma$ as acting on the labels of the vertices in the Dynkin diagram. For $1\leq i\leq \rank(\G)$, we denote by $\orb(i)$ and $\stab(i)$ the orbit of $i$ and, respectively, the stabilizer of $i$ under the action of the group generated by $\sigma$. We assume that the labelling of the Dynkin diagram of $\G$ is such that $1,\dots, n$ form a set of representatives for the orbits of $\sigma$. The integers $r_i=|\stab(i)|$, $1\leq i \leq n$, will play a role later on.

We denote by $\G^\sigma$ the fixed point subalgebra of $\G$. Note that $\G^\sigma$ is itself a simple Lie algebra of type $C_n$, $n\geq 3$, $B_n$, $n\geq 2$, or $F_4$, if $\G$ is of type $A_{2n-1}$, $n\geq 3$, $D_{n+1}$, $n\geq 2$, or $E_6$ and $m=2$, and $\G^\sigma$ is a simple Lie algebra of type $G_2$ if $\G$ is of type $D_4$ and $m=3$. In what follows $n$ will denote the rank of $\G^\sigma$.

We fix a Cartan subalgebra $\H$ of $\lie g$ and we denote the corresponding set of roots, the root lattice, and the weight lattice by $R_{\lie g}$, $Q_{\lie g}$, and $P_{\lie g}$, respectively. We also fix  a basis of $\G$ with respect to $\H$ and denote the corresponding set of positive roots, the $\bz_+$-cone spanned by the positive roots, and the cone of dominant weights by $R^+_{\lie g}$, $Q^+_{\lie g}$, and $P^+_{\lie g}$, respectively. Let $\lie n^\pm$ be the nilpotent subalgebra of $\G$ determined by $\pm R_\G^+$. 

The subalgebra $\lie h^\sigma=\lie g^\sigma\cap\lie h$ is a Cartan subalgebra of $\lie g^\sigma$ and we denote by $R$, $Q$, and $P$ the corresponding set of roots, the root lattice, and the weight lattice, respectively. Similarly, $\G^\sigma \cap\lie n^+$ is the nilpotent radical of a Borel subalgebra of $\G^\sigma$ and we denote the corresponding set of positive roots, the $\bz_+$-cone spanned by the positive roots, and the cone of dominant weights by $R^+$, $Q^+$, and $P^+$, respectively. Also, we fix basis $\{\a_1,\dots,\a_n\}$ of $R$ and a Chevalley basis $\{x_\alpha, ~h_i~:~\alpha\in R,~ 1\le i\leq n \}$ for $\G^\sigma$.

As in Section \ref{handh} of this paper, we identify $(\lie h^\sigma)^*$ with a subspace of $\lie h^*$ so that $P^+$ is a subset of $P^+_{\lie g}$ and the fundamental weights $\{\omega_i\}_{1\leq i\leq n}$ of $\lie g^\sigma$ are a subset of the fundamental weights $\{\omega_i\}_{1\leq i\leq \rank(\G)}$ of $\lie g$.  For $\lambda\in P^+$  we denote by $V(\lambda)$  the irreducible highest weight module for $\lie g^\sigma$  with highest weight $\lambda$ and for $\mu\in P^+_{\lie g}$ we denote by $V_{\lie g}(\mu)$ the irreducible highest weight module for $\lie g$ with highest weight $\mu$.


\subsection{}  Let $L(\lie g) =\lie g\otimes\bc[t, t^{-1}]$ be the loop algebra of $\lie g$ with the usual Lie bracket, given by the $\bc[t,t^{-1}]$--bilinear extension of the Lie bracket on $\lie g$, and denote by $\G[t]$ the current algebra associated to $\G$.  Extend $\sigma$ to an order $m$ automorphism of $L(\lie g)$ defined by $$\sigma(x\otimes t^i)=\sigma(x)\otimes e^{-2\pi i/m}t^i.$$

Modulo the center and the scaling element, the special maximal parabolic subalgebra of an affine Lie algebra of type $A_{2n-1}^{(2)}$, $n\geq 3$, $D_{n+1}^{(2)}$, $n\geq 2$, $E_6^{(2)}$, or $D_4^{(3)}$ is isomorphic to the fixed point subalgebra $\G[t]^\sigma$ of $\G[t]$, for $\G$  of type $A_{2n-1}$, $n\geq 3$, $D_{n+1}$, $n\geq 2$, $E_6$ and $m=2$, and  of type $D_4$ and $m=3$, respectively. Remark that both $\G[t]$ and $\G[t]^\sigma$ are naturally $\bz_+$-graded and the homogeneous components of degree zero are $\G$ and, respectively, $\G^\sigma$. We denote by $\ev_0:\lie g[t]^\sigma\to \lie g^\sigma$ the projection onto the homogeneous component of degree zero.

 Remark that the Lie  subalgebras  $\lie n^\pm[t]: =\lie n^\pm\otimes \bc[t]$ and $\lie h[t]: =\lie h^\pm\otimes \bc[t]$ are stabilized by $\sigma$ and we have a triangular decomposition $$\lie g[t]^\sigma=\lie n^-[t]^\sigma\oplus\lie h[t]^\sigma\oplus \lie n^+[t]^\sigma.$$ We denote by $\H[t]^\sigma_+$ the ideal of $\H[t]^\sigma$ generated by positive degree homogeneous elements.
 
Given $\boz= (z_1,\dots, z_k)\in(\bc^\times)^k$ and $N\geq 1$, let 
$$
\ev_{\boz, N}: L(\G)\to \oplus_{s=1}^N \G^{z_s,N}, \quad \G^{z_s,N}=\G\otimes \frac{\bc[t,t^{-1}]}{((t-z_s)^N)},
$$
be the canonical Lie algebra morphism and  let 
$$
\Psi_{\boz, N}: \lie g[t]^\sigma \hookrightarrow L(\G)\to \oplus_{s=1}^N \G^{z_s,N},
$$ be the restriction  of $\ev_{\boz, N}$ to $\lie g[t]^\sigma$. We have the following analog of Lemma \ref{generalev}.
\begin{lem} If $z_s\ne z_p$ for $1\le s\ne p\le k$ then $\ev_{\boz, N}$ is surjective. If  $z_s^m \ne  z_p^m$ for $1\le s\ne p\le k$, then $\Psi_{\boz, N}$ is surjective.
\end{lem}


\subsection{}  
Given $\lambda\in P^+$,  set $$P(\lambda)=\bu(\lie g[t]^\sigma)\otimes_{\bu(\lie g^\sigma)} V(\lambda).$$  
The global Weyl module $W(\lambda)$ is defined to to be the maximal  $\lie g[t]^\sigma$--module quotient of $P(\lambda)$ with weights contained in $\lambda- Q^+$. Let $v_\lambda$ denote a highest weight vector of $V(\lambda)$ and denote by  $w_\lambda$ be the image of $1\otimes v_\lambda$ in $W(\lambda)$. The module $W(\lambda)$ is 
generated by an element $w_\lambda$ with  relations 
$$
\lie n^+[t]^\sigma w_\lambda=0,\quad  h w_\lambda=\lambda(h)w_\lambda,\quad (x_{-\alpha}\otimes 1)^{\lambda(\alpha^\vee)+1} w_\lambda=0,
$$ 
for all $ h\in\lie h^{\sigma}$ and $\alpha\in R^+$,  where  $\alpha^\vee$ is the co-root of $\G^\sigma$ corresponding to $\alpha$.

There is a natural right action of  $\lie h[t]^\sigma$ on $ W(\lambda)$ and we let $\ba_\lambda$ be the quotient of $\bu(\lie h[t]^\sigma_+)$ by the ideal $$\Ann (w_\lambda)= \{u\in \bu(\lie h[t]^\sigma_+): uw_\lambda=w_\lambda u=0\}.$$  Then, $W(\lambda)$ is a $(\lie g[t]^\sigma,\ba_\lambda)$--bimodule.

Given a maximal ideal $\bi$ of $\ba_\lambda$, define the local Weyl module $W(\lambda,\bi)$  as 
$$
W(\lambda, \bi)=W(\lambda)\otimes_{\ba_\lambda}\ba_{\lambda}/\bi.
$$
We remark that the modules $W(\lambda)$  are naturally $\bz_+$-graded by assigning degree zero to $w_\lambda$. The ideal $\Ann(w_\lambda)$ is also a graded ideal and  in consequence the algebra $\ba_\lambda$ is a $\bz_+$-graded algebra. Let $\bi_{\lambda,0}$ be the maximal graded  ideal of $\ba_\lambda$. The local Weyl modules $W(\lambda,\bi_{\lambda,0})$ are the only graded local Weyl modules. 


\subsection{} The local Weyl modules for $L(\lie g)$ are indexed by  polynomials  $\bpi=(\pi_1,\dots, \pi_{\rank(\G)})$  in an indeterminate $u$ where  each  $\pi_i$ has constant term one and are subject to the notation in Section \ref{untwistedweyl}.  Furthermore, we denote by $\bvpi_i$ the $\rank(\G)$--tuple  of polynomials with entry $(1-u)$ in the $i^{\rm th}$ place and $1$ elsewhere.

The following analogues of Theorem \ref{charietal}  and Proposition \ref{piw} will be needed. Part (i) was proved in \cite{CL06} for $\G$ of type $A_n$ and in \cite{FoL07} for $\G$ simply-laced,  part (ii) was proved in \cite{CP01}, part (iii) was proved in \cite[Section 2.7]{CFS08}, and the proof of part (iv) is the same as the proof of Proposition \ref{piw}.
\begin{thm} Let $\bpi=(\pi_1,\dots, \pi_{\rank(\G)})$ be as above and let $W(\bpi)$ be the corresponding local Weyl module for $L(\G)$. Then,
\begin{enumerit}
\item[(i)] We have,  $$\dim W(\bpi)=\prod_{i=1}^n \dim W(\bvpi_i)^{\deg\pi_i}.$$ 

\item[(ii)] The unique irreducible quotient $V(\bpi)$ of $W(\bpi)$ is isomorphic to $\ev_{\boz, 1}^*(V_{\lie g}(\lambda_1)\otimes\cdots\otimes V_{\lie g}(\lambda_k))$ where 
 $$\lambda_s=\sum_{i=1}^{\rank{\G}}m_{i,s}\omega_i,\quad 1\leq s\leq k, \quad\text{and}\quad \pi_i=\prod_{s=1}^k(1-z_su)^{m_{i,s}}, \quad 1\leq i\leq \rank(\G).$$

\item[(iii)]  For $N\in\bn$ sufficiently large, we have $$\left(\lie g\otimes \prod_{s=1}^k(t-z_s)^N\bc[t,t^{-1}]\right)W(\bpi)=0.$$

\item[(iv)] Let  $z_1,\dots, z_k$ be  the distinct roots of  $\pi_1\cdots\pi_n$ and assume that $z_s^m\ne z_p^m$ for $1\le s\ne p\le k$. Let 
$$
\lambda=\sum_{i=1}^n \left(\sum_{j\in\orb(i)}\deg\pi_{j}\right)\omega_i.
$$
There exists a maximal ideal $\bi_\bpi$ in $\ba_\lambda$ such that $\Psi^*_{\boz, N} W(\bpi)$ is a quotient of $W(\lambda,\bi_\bpi)$. In consequence, $$\dim W(\lambda,\bi_\bpi)\ge \Psi^*_{\boz, N} W(\bpi).$$
\end{enumerit}
\end{thm}


\subsection{}  The arguments in Sections \ref{five} and \ref{six}  can now emulated to establish the following result.

\begin{thm} \label{thm8}
Let $\lambda\in P^+$. The algebra $\ba_\lambda$ is a graded polynomial algebra in variables $T_{i,r}$ of grade $r_ir$, $1\le i\le n$ and $1\le r\le\lambda(\alpha_i^\vee)$. The module $W(\lambda)$ is a finitely generated $\ba_\lambda$-module and for any maximal ideal $\bi$ of $\ba_\lambda$, we have
$$\dim W(\lambda, \bi_{\lambda,0})\geq \dim W(\lambda, \bi)\geq \prod_{i=1}^n \dim W(\bvpi_i)^{\lambda(\alpha_i^\vee)}.$$
\end{thm}

\subsection{}  Finally, the connection between the graded local Weyl modules and the level one Demazure modules for the corresponding affine twisted Lie algebra can be established in the same fashion as described in Section \ref{seven}. In fact, this was already done: for the following statement, part (i) was proved in \cite[Lemma 5.3.]{FK11}, part (ii) was proved in \cite[Theorem 5.1.]{FK11}, and part (iii) was proved in \cite[Theorem 2]{FoL06}.

\begin{thm}\label{gradetwist}  With the notation above, we have
\begin{enumerit}
\item[(i)]   For $1\le i\le n$ we have $$\dim W(\omega_i,\bi_{\omega_i,0})=\dim W(\bvpi_i).$$
\item[(ii)] Let $\lambda\in P^+$. The graded local Weyl module $W(\lambda,\bi_{\lambda,0})$ and the Demazure module $D(-\lambda)$ are isomorphic to as graded $\lie g[t]^\sigma$-modules.
\item[(iii)]  Let $\lambda\in P^+$.  Then, $$\dim D(-\lambda)=\prod_{i=1}^n\dim D(-\omega_i)^{\lambda(\alpha_i^\vee)}.$$\end{enumerit}
\end{thm}
As a consequence of Theorem \ref{thm8} and Theorem \ref{gradetwist} we obtain the following.
\begin{thm} \label{thm10}
Let $\lambda\in P^+$. For any maximal ideal $\bi$ of $\ba_\lambda$, we have 
$$
\dim W(\lambda,\bi)=\prod_{i=1}^n\dim W(\omega_i,\bi_{\omega_i,0})^{\lambda(\alpha_i^\vee)}.
$$
\end{thm}
To conclude, we have proved the following.
\begin{thm} \label{thm11}
Let $\lambda\in P^+$. The global Weyl module $W(\lambda)$ is a free $\ba_\lambda$-module of rank equal to 
$$\prod_{i=1}^n\dim W(\omega_i,\bi_{\omega_i,0})^{\lambda(\alpha_i^\vee)}.$$
\end{thm}

\begin{rem}
Although, in this section, we have excluded from consideration the case of $\G$ of type $A_{2n}$ most of the statements are valid in that situation too. The first statement that fails to be true in this case is Theorem \ref{gradetwist}(ii), and in fact, for $\G$ of type $A_{2n}$, the dimension of a local Weyl module $W(\lambda,\bi_{\lambda,0})$ for $\G[t]^\sigma$ is not known if $\lambda(\alpha_n^{\vee})$ is even.  There are examples in \cite{FK11} which show that the dimension of a graded  local Weyl module can be bigger than the dimension of the corresponding Demazure module.
\end{rem}

\bibliographystyle{plain}
\bibliography{cik-biblist}
\end{document}